\documentclass[12pt, reqno]{amsart}
\usepackage{amsmath}
\usepackage{amssymb}
\usepackage{epsfig}
\usepackage{graphicx}
\usepackage{color}
\usepackage{fullpage}
\definecolor{shadecolor}{gray}{0.875}
\usepackage{amscd}
\usepackage{mathtools}
\usepackage{hyperref}

\numberwithin{equation}{section}

\input xy
\xyoption{all}

\calclayout
\allowdisplaybreaks[3]

\theoremstyle{plain}
\newtheorem{prop}{Proposition}[section]

\newtheorem{theo}[prop]{Theorem}
\newtheorem{coro}[prop]{Corollary}

\newtheorem{lemm}[prop]{Lemma}

\theoremstyle{definition}
\newtheorem{defi}[prop]{Definition}

\newtheorem{rema}[prop]{Remark}

\newtheorem{exam}[prop]{Example}

\newtheorem{clai}[prop]{Claim}

\def\ra{\rightarrow}

\def\bR{{\mathbb R}}
\def\bZ{{\mathbb Z}}

\def\Eff{\overline{\mathrm{Eff}}}

\def\Aut{\mathrm{Aut}}

\def\Hilb{\mathrm{Hilb}}

\def\ch{\mathrm{char}}
\def\Spec{\mathrm{Spec}}

\makeatother
\makeatletter

\author{Roya Beheshti}
\address{Department of Mathematics \\
Washington University in St.~Louis \\
St. Louis, MO \, \, 63130}
\email{beheshti@wustl.edu}

\author{Brian Lehmann}
\address{Department of Mathematics \\
Boston College  \\
Chestnut Hill, MA \, \, 02467}
\email{lehmannb@bc.edu}

\author{Eric Riedl}
\address{Department of Mathematics \\
University of Notre Dame  \\
255 Hurley Hall \\
Notre Dame, IN 46556}
\email{eriedl@nd.edu}

\author{Sho Tanimoto }
\address{Graduate School of Mathematics, Nagoya University, Furocho Chikusa-ku, Nagoya, 464-8602, Japan}
\address{Institute for Advanced Research, Nagoya University}
\email{sho.tanimoto@math.nagoya-u.ac.jp}

\title[Rational curves on del Pezzo surfaces]{Rational curves on del Pezzo surfaces\\ in positive characteristic}

\begin{document}
\date{\today}

\begin{abstract}
We study the space of rational curves on del Pezzo surfaces in positive characteristic.  For most primes $p$ we prove the irreducibility of the moduli space of rational curves of a given nef class, extending results of Testa in characteristic $0$.  We also investigate the principles of Geometric Manin's Conjecture for weak del Pezzo surfaces. In the course of this investigation, we give examples of weak del Pezzo surfaces defined over $\mathbb F_2(t)$ or $\mathbb{F}_{3}(t)$ such that the exceptional sets in Manin's Conjecture are Zariski dense.
\end{abstract}

\maketitle

\section{Introduction}

Let $S$ be a del Pezzo surface over an algebraically closed field $k$. Let $\overline{M}_{0,0}(S)$ be the Kontsevich moduli space of stable maps of genus $0$ and let $\overline{\mathrm{Rat}}(S)$ denote the union of the irreducible components of $\overline{M}_{0,0}(S)$ which generically parametrize stable maps from irreducible domains. (Here we endow each component with its reduced structure.)  We are interested in the ``discrete'' invariants of $\overline{\mathrm{Rat}}(S)$: the number of irreducible components of a given degree, the dimension of the components, and so on.

In characteristic $0$, the behavior of these invariants is predicted by Geometric Manin's Conjecture as formulated in \cite{LTCompos}.  \cite{Testathesis} and \cite{Testa09} classified the components of $\overline{\mathrm{Rat}}(S)$ for ``most'' del Pezzo surfaces; in particular, Testa's work verifies Geometric Manin's Conjecture for such surfaces.  (As a secondary result, in this paper we extend Testa's classification to all del Pezzo surfaces in characteristic $0$.)

Our main focus is del Pezzo surfaces in characteristic $p$.   In particular, we would like to analyze whether the framework of Geometric Manin's Conjecture can be extended to cover such surfaces.  We classify the components of $\overline{\mathrm{Rat}}(S)$ for ``most'' del Pezzo surfaces in characteristic $p$ and verify that the principles of Geometric Manin's Conjecture hold in these examples.

\subsection{Summary of main results}
Our first statement addresses the components of $\overline{\mathrm{Rat}}(S)$ which have larger than the expected dimension.  We show that most weak del Pezzo surfaces do not carry any dominant families of this type.

\begin{theo} \label{theo:maintheorem1}
Let $S$ be a weak del Pezzo surface over an algebraically closed field such that a general member of $|-K_{S}|$ is smooth.  Then the only components of $\overline{\mathrm{Rat}}(S)$ which have greater than the expected dimension will parametrize multiple covers of rational curves $C$ satisfying $-K_{S} \cdot C \leq 1$.  In particular, there are no dominant families of rational curves with larger than the expected dimension.
\end{theo}

The weak del Pezzo surfaces $S$ for which Theorem \ref{theo:maintheorem1} does not apply -- that is, the surfaces $S$ such that every member of $|-K_{S}|$ is singular -- are classified by \cite{KN20}.  There are 3 infinite families and 11 sporadic examples; these examples only occur in characteristic $2$ or $3$ and when $S$ has degree at most $2$.   Note that for such a surface $S$ the curves in $|-K_{S}|$ are all rational and thus Theorem \ref{theo:maintheorem1} must fail.  We will discuss these examples in more depth in Section \ref{subsec:examples} and Section \ref{sect:increaseainv}.

Under more restrictive conditions, we show the stronger result that there are no dominant irreducible components of $\overline{\mathrm{Rat}}(S)$ yielding an inseparable family of rational curves.

\begin{theo} \label{theo:maintheorem2}
Let $S$ be a smooth del Pezzo surface of degree $d$ over an algebraically closed field $k$ of characteristic $p$. Assume that either $d \geq 2$ or $d = 1$ and $p \geq 11$.   When $d = 3$, we assume furthermore that $S$ is not the following exception:
\begin{enumerate}
\item $\ch(k)=2$ and $S$ is the Fermat cubic surface $x^{3}+y^{3}+z^{3} + w^3=0$.
\end{enumerate}
When $d = 2$, we assume furthermore that $S$ is not one of the following list of exceptions:
\begin{enumerate}
\setcounter{enumi}{1}
\item $\ch(k)=3$ and $S$ is the double cover of $\mathbb{P}^{2}$ ramified along the Klein quartic curve, i.e., the curve defined by $zx^{3} + xy^{3} + yz^{3}=0$.
\item $\ch(k)=2$ and $S$ is a double cover of $\mathbb{P}^{2}$ defined by the equation $w^{2} + wy^2 + g_{4}$ where $g_4$ is a homogeneous polynomial in $x, y, z$.
%\item $\ch(k)=2$ and $S$ is the blow-up of the Fermat cubic surface.
\end{enumerate}
Then every dominant component of $\overline{\mathrm{Rat}}(S)$ is separable and generically parametrizes free rational curves.
\end{theo}

\begin{rema}
The surfaces described in the exceptions above do actually contain a dominant inseparable family of rational curves.
\end{rema}

\begin{rema}
It is interesting to note that the exceptional cases in Theorem \ref{theo:maintheorem2} are exactly the del Pezzo surfaces of degree $\geq 2$ that are not Frobenius split.  \cite[Example 5.5]{Hara98} shows that the Fermat cubic surface in characteristic $2$ is the unique smooth cubic surface that is not F-split.  \cite[Theorem 0.3]{Saito17} shows that the smooth degree 2 del Pezzos that are not F-split are the double cover of $\mathbb{P}^{2}$ branched over the Fermat quartic in characteristic $3$ and the double covers of $\mathbb{P}^{2}$ branched over the double line in characteristic $2$.  (The Fermat quartic in $\mathbb{P}^{2}$ in characteristic $3$ is projectively equivalent to the equation given above by \cite[Proposition 3.7]{Pardini86}. See also \cite[Formula (1.11)]{Elkies}.)
\end{rema}

If we impose some further restrictions on the characteristic then we can completely classify components of $\overline{\mathrm{Rat}}(S)$.  Let $\delta(d)$ be the function defined by
\begin{equation}
\label{delta}
\delta(d) = 
\begin{cases}
2 & \text{ if $d \geq 4$}\\
3 & \text{ if $d = 2, 3$}\\
11 & \text{ if $d = 1$}.
\end{cases}
\end{equation}
The following statement extends results of \cite{Testathesis} and \cite{Testa09} to positive characteristic: 

\begin{theo} \label{theo:maintheorem3}
Let $S$ be a smooth del Pezzo surface of degree $d$ over an algebraically closed field $k$ of characteristic $p$. Assume that $p \geq \delta(d)$. Furthermore when $d = 2$, we assume that $S$ is not isomorphic to the surface listed in Theorem~\ref{theo:maintheorem2}.(2).

Let $\beta$ be a nef class on $S$ satisfying $-K_S \cdot \beta \geq 3$.  Then:  
\begin{itemize}
\item If $\beta$ is not a multiple of a $-K_{S}$-conic, then there is a unique component $M$ of $\overline{M}_{0,0}(S,\beta)$ generically parametrizing stable maps with irreducible domains.  The general map parametrized by $M$ is a birational map onto a free curve.
\item If $\beta$ is a multiple of a smooth rational conic, then there is a unique component $M$ of $\overline{M}_{0,0}(S,\beta)$ generically parametrizing stable maps with irreducible domains.  The general map parametrized by $M$ is a finite cover of a smooth conic.
\item If $d=2$ and $\beta$ is a multiple of $-K_{S}$, or $d = 1$ and there is a contraction of a $(-1)$-curve $\phi: S \to S'$ such that $\beta$ is a multiple of the  pullback of $-K_{S'}$, then there are exactly two components of $\overline{M}_{0,0}(S,\beta)$ parametrizing stable maps with irreducible domains.  One component generically parametrizes birational maps onto free curves, the other generically parametrizes multiple covers of $-K_S$-conics.
\item If $d=1$ and $\beta$ is a multiple of $-2K_{S}$, then there are at least two components of $\overline{M}_{0,0}(S,\beta)$ parametrizing stable maps with irreducible domains.  There is a unique component that generically parametrizes birational maps onto free curves, and the other components generically parametrize multiple covers of $-K_S$-conics.
\end{itemize}
\end{theo}

Finally, we extend the results of \cite{Testathesis} to finish the classification of components of $\overline{\mathrm{Rat}}(S)$ of anticanonical degree $\geq 3$ on arbitrary del Pezzo surfaces $S$ of degree $1$ in characteristic $0$.  Theorem \ref{theo:dp1char0classification} finishes the proof of the analogue of Theorem \ref{theo:maintheorem3} in characteristic $0$; in particular:

\begin{theo}
\label{theo:degree1char0}
Let $S$ be a smooth del Pezzo surface of degree $1$ over an algebraically closed field of characteristic $0$.  Let $\alpha$ be a nef curve class on $S$ satisfying $-K_{S} \cdot \alpha \geq 3$ which is not a multiple of a smooth rational conic.  Then $\overline{\mathrm{Rat}}(S)$ contains a unique component generically parametrizing birational maps onto free curves of numerical class $\alpha$.
\end{theo}

\subsection{Geometric Manin's Conjecture}  
In his unpublished notes \cite{Bat88}, Batyrev developed a heuristic for Manin's conjecture for Fano varieties over finite fields.  (This heuristic inspired Batyrev-Manin's conjecture over number fields as formulated in the series of papers \cite{BM}, \cite{Peyre}, \cite{BT98}, \cite{LST18}.)  Batyrev's heurstic relies on several geometric assumptions about the structure of the moduli space of curves on a Fano variety.  When working over an algebraically closed field of  characteristic $0$, these assumptions were further revised and were systematized as Geometric Manin's Conjecture in \cite{LTCompos}.

Our main motivation for this paper is to test whether the principles of Geometric Manin's Conjecture hold for surfaces in characteristic $p$.  In brief, Geometric Manin's Conjecture predicts that the discrete invariants of $\overline{\mathrm{Rat}}(S)$ -- the number of components and their dimensions -- are controlled by a geometric quantity known as the Fujita invariant. 

\begin{defi}
Suppose that $X$ is a smooth projective variety over an algebraically closed field equipped with a nef divisor $L$.  The Fujita invariant $a(X,L)$ is defined as follows.  If $L$ is not big, we set $a(X,L) = \infty$.  Otherwise, we define
$$
a(X, L) := \min \{ t\in \bR \mid K_{X} + tL \textrm{ is pseudo-effective} \}.
$$
\end{defi}

\begin{defi}
Let $X$ be a smooth weak Fano variety over an algebraically closed field.  We say that a generically finite morphism $f: Y \to X$ from a smooth projective variety $Y$ is a breaking map if $a(Y,-f^{*}K_{X}) > a(X,-K_{X})$.
\end{defi}

Suppose that $X$ is a weak Fano variety over an algebraically closed field of characteristic $0$.  \cite[Theorem 1.1]{LTCompos} proves that for any component $M$ of $\overline{\mathrm{Rat}}(X)$ with larger than the expected dimension, there is  a breaking map $f: Y \to X$ and a component $N$ of $\overline{\mathrm{Rat}}(Y)$ such that pushforward under $f$ maps $N$ birationally onto $M$.

The techniques used to prove \cite[Theorem 1.1]{LTCompos} do not work in characteristic $p$.  The main new obstruction is the existence of inseparable maps. On the one hand, inseparable maps provide new ``unexpected'' examples of dominant breaking maps.  On the other hand, inseparable maps provide new ``unexpected'' families of rational curves.  The key question is whether such phenomena match up to preserve the relationship between the two.

The following theorem verifies this correspondence for surfaces in characteristic $p$.

\begin{theo} \label{theo:breakingandhigherdim}
Let $S$ be a weak del Pezzo surface over an algebraically closed field of characteristic $p$.  Then the following are equivalent:
\begin{enumerate}
\item $S$ admits a dominant family of rational curves with larger than the expected dimension.
\item There is a dominant breaking map $f: Y \to S$.
\end{enumerate}
\end{theo}

\begin{rema}
Note that Theorem \ref{theo:breakingandhigherdim} is slightly weaker than \cite[Theorem 1.1]{LTCompos} because it does not address whether each family of rational curves with larger than expected dimension factors through a breaking map.  We expect this to be the case.
\end{rema}

\begin{rema}
The example of Shioda hypersurfaces suggests that in higher dimensions the correct interpretation of Geometric Manin's Conjecture may be more subtle.  We plan to return to this question in future work.
\end{rema}

In fact, for a weak del Pezzo surface $S$ we can completely classify all breaking maps $f: Y \to S$.  The existence of such maps is closely related to the geometry of the linear series $|-K_{S}|$ and $|-2K_{S}|$.

\begin{theo} \label{theo:intro_dominantaclassificationsurfaces}
Let $S$ be a weak del Pezzo surface of degree $d$ and suppose that $f: Y \to S$ is a dominant generically finite morphism such that $a(Y,-f^{*}K_{S}) > a(S,-K_{S})$.  Then we are in one of the following situations:
\begin{enumerate}
\item $\ch(k) = 2$ or $3$, $d =1$, and $f$ is birationally equivalent to the base change of a quasi-elliptic fibration by a non-separable map to the target curve.
In this case $a(Y,-f^{*}K_{S}) = 2$ and the linear series $|-K_{S}|$ defines the quasi-elliptic fibration (after blowing-up the base point).
\item $\ch(k)=2$, $d=2$, and $f$ is birationally equivalent to a purely inseparable morphism of degree $2$ from $\mathbb{P}^{2}$ to the anticanonical model of $S$.
In this case, we have $a(Y, -f^*K_S) =  3/2$ and $|-K_S|$ defines a purely inseparable degree $2$ cover.
\item $\ch(k)=2$, $d = 1$, and $f$ is birationally equivalent to a purely inseparable morphism of degree $2$ from the quadric cone $Q$ to the anticanonical model of $S$.
In this case, we have $a(Y, -f^*K_S) = 2$ and $|-2K_S|$ defines a purely inseparable degree $2$ cover.
\item $\ch(k)=2$, $d = 1$, and $f$ is birationally equivalent to a non-separable morphism of degree $4$ from $\mathbb{P}^{2}$ to the anticanonical model of $S$.
In this case, we have $a(Y, -f^*K_S) = 3/2$ and  $|-2K_S|$ defines a purely inseparable degree $2$ cover.
\end{enumerate}
When $S$ is a del Pezzo surface then none of (1)-(4) can occur.
\end{theo}

As shown by Theorem \ref{theo:breakingandhigherdim}, the possible weak del Pezzo surfaces $S$ in Theorem \ref{theo:intro_dominantaclassificationsurfaces} are the same as the weak del Pezzo surfaces classified by \cite[Theorem 1.4]{KN20}.  These examples are quite interesting; they show that over a function field there can be a Zariski dense set of rational points which outpaces the exponential term in the rate predicted by Manin's Conjecture.  (Conjecturally there is no Zariski dense set with the analogous property over a number field.  From the perspective of the Fujita invariant, this discrepancy is predicted by the fact that \cite[Theorem 1.1]{HJ16} fails in positive characteristic even in dimension $2$. This is in contrast to positive results for surfaces obtained in \cite{LTT14} and \cite{LT16}.)   

\begin{exam}[{\cite[Table 2]{KN20b}}]
\label{exam:intro_halpen}
Let the ground field $k$ be $\mathbb F_3$.
Let $S'$ be the surface in the weighted projective space $\mathbb P(1, 1, 2, 3)_{(x:y:z:w)}$ defined by 
\[
w^2 + z^3 - x^2y^2(x+y)^2 = 0.
\]
Then this is a du Val del Pezzo surface with four $A_2$ singularities. (\cite[Table 2]{KN20b})
We denote its minimal resolution by $S$ so that $S$ is a weak del Pezzo surface of degree $1$.

Let $\beta : \widetilde{S} \to S$ be the blow up of the base point for $|-K_S|$. Then $|-K_{\widetilde{S}}|$ defines a quasi-elliptic fibration, i.e., a fibration $\pi : \widetilde{S} \to B = \mathbb P^1$ such that a general fiber is a cuspidal rational curve. To construct the component of $\overline{M}_{0,0}(S)$ parametrizing fibers of $\pi$, one needs to take a purely inseparable base change by the Frobenius map $F : B' = \mathbb P^1 \to B = \mathbb P^1$. Set $Y = \widetilde{S} \times_B B'$ and let $\widetilde{Y} \to Y$ be the normalization map. We denote by $\rho: \widetilde{Y} \to B'$ the induced generically smooth fibration and by $f: \widetilde{Y} \to S$ the induced inseparable generically finite map. 

Now we take the base change to the field $K = \mathbb F_3(t)$.  We are interested in the asymptotic growth of the number of $K$-rational points of $S_{K}$ of bounded height.  We claim that the images of the points on $\widetilde{Y}_{K}$ under $f_{K}$ yield a Zariski dense set which grows faster than the expected growth rate (even in the exponential term).

Let $C_K$ be a geometrically integral fiber of $\rho_K: \widetilde{Y}_K \to B'_{K}$ defined over $K$ so that $C_{K} \cong \mathbb{P}^{1}_{K}$. Note that we have $a(C_K, -f_{K}^{*}K_{S_K}) = 2 > a(S_K, -K_{S_K}) = 1$.  Thus the points on each such fiber $C_{K}$ grow at faster than the expected rate: there will be $\sim q^{2d}$ points on $C_{K}$ of anticanonical height $\leq d$ compared to the ``expected'' number $\sim q^{d} d^{\rho(S_{K})-1}$ for $S_{K}$.

Since $B'_K \cong \mathbb{P}^{1}_K$ there will be a Zariski dense set of deformations of $C_K$ defined over the ground field. Also note that $C_K(K)$ is Zariski dense as $C_K$ is isomorphic to $\mathbb P^1_K$. Thus we need to remove a Zariski dense set of rational points on $S_K$ in order to obtain the desired growth rate for rational points.
(Although the set we must remove is Zariski dense, it is a thin set since it comes from $f_K :  Y_K  \to S_K$.)
\end{exam}

\begin{exam} \label{exam:CasciniTanaka} 
We work over $\mathbb{F}_2$.  We will recall an example of a surface considered in \cite[end of Section 9]{KM99} and in \cite{CT18} due to the failure of the Kawamata-Viehweg vanishing theorem.

Suppose we blow-up $\mathbb{P}^{2}$ at all seven $\mathbb{F}_{2}$-points.  We will obtain a weak del Pezzo surface $S$ of degree $2$.  \cite[Proposition 5.3]{CT18} shows that the $(-2)$-curves on $S$ will be precisely the strict transforms of the seven $\mathbb{F}_{2}$-lines on $\mathbb{P}^{2}$.  \cite[Theorem 4.1]{CT18} shows that $|-K_{S}|$ defines a purely inseparable degree $2$ map to $\mathbb{P}^{2}$.  This map factors through the anticanonical model $S'$ of $S$ which has seven $A_{1}$-singularities.

Let $w^2 = f_4(x, y, z)$ be the defining equation of $S'$ in the weighted projective space $\mathbb P(1,1, 1, 2)$ where $f_4$ is a homogenous polynomial of degree $4$.
By construction $f_4$ has coefficients in $\mathbb F_2$. We define the morphism
\[
f : \mathbb P^2 \to S': (s:t:u) \mapsto (x : y : z: w) = (s^2: t^2: u^2 : f_4(s, t, u)).
\]
Then the Frobenius map $F : \mathbb P^2 \to \mathbb P^2$ factors through $f$.

Since $-K_{S'}$ is the pullback of $\mathcal{O}(1)$ under the map $S' \to \mathbb{P}^{2}$, we see that $-f^{*}K_{S'} = \mathcal{O}(2)$.  Thus $a(\mathbb P^2,-f^{*}K_{S'}) = \frac{3}{2}$ while $a(S', -K_{S'}) = a(S,-K_{S}) = 1$. Again working over $K = \mathbb F_2(t)$, the exceptional set for $S_K$ must contain a Zariski dense subset of rational points $f_{K}(\mathbb P_{K}^2(K))$ which have an asymptotic growth rate of $q^{\frac{3}{2}d}$.
\end{exam}

\subsection{Our methods}
To prove Theorem~\ref{theo:maintheorem1}, first we check that components parametrizing stable maps of anticanonical degree $\leq 2$ have expected dimension by classifying these low degree rational curves. (Lemma~\ref{lemm:(-2)-curve}, Lemma~\ref{lemm:lines}, Lemma~\ref{lemm:degree2curves}) Then by employing Bend and Break argument (Lemma~\ref{lemm:weakmbbfordpsurfaces}), we prove that every dominant component of $\overline{\mathrm{Rat}}(S)$ has expected dimension by using an inductive proof on the degree of rational curves. (Proposition~\ref{prop:expecteddim}). 

 A proof of Theorem~\ref{theo:maintheorem2} is similar. We first analyze separability of families of $-K_S$-conics and cubics in Section~\ref{sect:lowdegree} then we use Bend and Break argument (Lemma~\ref{lemm:weakmbbfordpsurfaces}) to prove that every dominant component of $\overline{\mathrm{Rat}}(S)$ parametrizes a free rational curve. (Proposition~\ref{prop: free}).

To achieve Theorem~\ref{theo:maintheorem3}, we first prove that under the assumption on $\mathrm{char}(k)$ every dominant family of rational curves of anticanonical degree $\leq 3$ contains a free rational curve using some deformation theory of rational curves in positive characteristic proved in \cite{IIL20}. (Theorem~\ref{theo:separabilitylargechar}) To this end, one needs to bound the arithmetic genus of rational curves of low degree on a del Pezzo surface $S$ and this is the main reason why our assumption on the characteristic of the ground field depends on the degree of $S$. Then we look at 1-dimensional loci of stable maps of anticanonical degree $e$ passing through $e-2$ general points, and prove that these loci are contained in the smooth locus of $\overline{M}_{0,0}(S)$ using an inductive argument whose base case is settled by Theorem~\ref{theo:separabilitylargechar}. (Lemma~\ref{curve-general_p}) Finally we lift everything to characteristic $0$ and use a specialization argument combined with \cite{Testa09} to conclude our main theorem. Theorem~\ref{theo:degree1char0} is also obtained using a similar idea: the irreducibility is known for general del Pezzo surfaces of degree $1$ by \cite{Testathesis} and we use a specialization argument to obtain the main theorem. We believe that this specialization argument is new, and it has potential to be applicable to problems on the space of rational curves on other Fano varieties.

Finally to obtain Theorem~\ref{theo:intro_dominantaclassificationsurfaces}, we use the 2 dimensional Minimal Model Program and classify smooth projective polarized surfaces with higher $a$-invariants (Theorem~\ref{theo:largeainvsurfaces}). Then we use this result to deduce Theorem~\ref{theo:intro_dominantaclassificationsurfaces}. Finally we found examples of weak del Pezzo surfaces satisfying Theorem~\ref{theo:intro_dominantaclassificationsurfaces} in \cite{KN20} and \cite{KN20b} which classified pathological examples of Du Val del Pezzo surfaces.
Theorem~\ref{theo:breakingandhigherdim} follows from Theorem~\ref{theo:intro_dominantaclassificationsurfaces} and the analysis of low degree rational curves in Section~\ref{sect:expdimlowdegree}.

\subsection{Previous works}

There is a vast literature studying the space of rational curves on various Fano varieties in characteristic $0$. The most relevant results to this paper are \cite{Testathesis} and \cite{Testa09} which classified components of $\overline{\mathrm{Rat}}(S)$ for most del Pezzo surfaces $S$ in characteristic $0$. \cite{LT21} addressed this problem for curves of genus $\geq 1$ on del Pezzo surfaces.  Readers interested in other classification results should consult \cite{BLRT20} and the references therein. 

Let us focus on results in positive characteristic. First of all, there are many papers which study the separable rational connectedness of smooth Fano varieties in characteristic $p$ (for example \cite{Shen10}, \cite{Zhu11}, \cite{CZ14}, \cite{GLPSTZ15}, \cite{Tian15}, \cite{CR19}, \cite{ST19}, and \cite{CS21}). J\'anos Koll\'ar asked whether any smooth Fano variety is separably rationally connected, but this question is wide open at this moment.  On the other hand there are only a few results on the classification of irreducible components of moduli spaces of rational curves on Fano varieties in characteristic $p$. \cite{BS18} discussed the irreducibility of moduli spaces of rational curves on low degree hypersurfaces in positive characteristic using a function field version of the circle method. Moduli spaces of rational curves on toric varieties are classified by Bourqui in \cite{Bourqui} using the Cox ring method. 

Examples of (weak) Fano varieties such that the exceptional set for Manin's conjecture is Zariski dense are well-documented over number fields. The first example was found by Batyrev and Tschinkel in \cite{BT96}, and recently more examples have been found and proved in \cite{LRud} and \cite{BHB20}. \cite{LST18} proposed a geometric description of these exceptional sets over number fields and proved that they are thin sets using the Minimal Model Program. The analogue of \cite{LRud} in positive characteristic has been studied in \cite{Man19}.

\

\noindent
{\bf Notation:} We will work throughout over a field denoted by $k$; usually $k$ will be algebraically closed.  A variety over $k$ is an integral separated scheme of finite type over $k$. 

For a smooth projective variety $X$ over $k$, $N^1(X)_{\bZ}$ denotes the space of divisors up to numerical equivalence and $N_1(X)_\bZ$ denotes the space of integral $1$-cycles up to numerical equivalence. For a projective morphism $f : X \to Y$ of schemes, $N^1(X/Y)_{\bZ}$ denotes the relative numerical N\'eron-Severi group of $X$ over $Y$.

For a scheme $X$, a component of $X$ means an irreducible component of $X$ endowed with its reduced structure.

Let $X$ be a smooth projective variety over $k$ and $L$ be an ample line bundle on $X$. An $L$-line ($L$-conic, or $L$-cubic) is a birational stable map $f : \mathbb P^1 \to X$ such that $\deg f^*L = 1$ (resp. $= 2$, or $3$).

\

\noindent
{\bf Acknowledgements:}
The authors would like to thank Jason Starr for explaining the deformation theory of stable maps in characteristic $p$. The authors would also like to thank Damiano Testa for discussing his results \cite{Testathesis} and \cite{Testa09}.
The authors would like to thank Yusuke Nakamura for answering our questions regarding del Pezzo surfaces in positive characteristic and Tim Browning for letting us know about the results in \cite{BS18} on rational curves on low degree smooth hypersurfaces in characteristic $p$. The authors would like to thank Damiano Testa for comments on an early draft of this paper.

This project was started at the SQuaRE workshop ``Geometric Manin's Conjecture in characteristic $p$'' at the American Institute of Mathematics. The authors would like to thank AIM for the generous support.

Roya Beheshti was supported by NSF grant 2101935. Brian Lehmann was supported by NSF grant 1600875. Eric Riedl was supported by NSF CAREER grant DMS-1945944.  Sho Tanimoto was partially supported by Inamori Foundation, by JSPS KAKENHI Early-Career Scientists Grant number 19K14512, by JSPS Bilateral Joint Research Projects Grant number JPJSBP120219935, by MEXT Japan, Leading Initiative for Excellent Young Researchers (LEADER), and by JST FOREST program Grant number JPMJFR212Z.

\section{Preliminaries}

We restrict our attention to dimension $\leq 3$.  In these dimensions we have resolutions of singularities  over any perfect field of characteristic $p$ by \cite{Abh56}, \cite{CP08}, and \cite{CP09}. We also can run the Minimal Model Program freely in dimension $2$ by \cite{Mumford}, \cite{BM77}, and \cite{BM76}, and in dimension $3$ if $p > 5$ by \cite{HX15}, \cite{CTX15}, \cite{Bir16}, and \cite{BW17}.

\subsection{Classes of singularities}
\label{subsec:node_cusp}

Let $p$ be a closed point of a reduced (possibly reducible) curve $C$ over an algebraically closed field $k$.  We say that $p$ is a node if we have a formal-local equivalence
\begin{equation*}
\widehat{\mathcal{O}_{C,p}} \cong k[[x,y]]/(xy)
\end{equation*}
If $f: Z \to X$ is a stable map which is birational onto its image and the image is a nodal curve then the normal sheaf $N_{f/Z}$ is locally free.

We say that $p$ is a cusp if $C$ is unibranch at $p$ and we have a formal-local equivalence
\begin{equation*}
\widehat{\mathcal{O}_{C,p}} \cong k[[x,y]]/(y^{2} + g_{3}(x,y))
\end{equation*}
for some homogeneous cubic $g_{3}$. If the characteristic is not $2$, then every cusp is formally-locally equivalent to the cusp defined by $y^{2} = x^{3}$.  
If the characteristic is equal to $2$, then the family of cusps has moduli.

Suppose that $f: Z \to X$ is a birational map from an irreducible smooth curve $Z$ and $p \in Z$ maps to a cusp in $f(Z)$.  If the characteristic is not $2$, then the normal sheaf $N_{f/X}$ has a torsion subsheaf of length $1$ at $p$.  If the characteristic is equal to $2$, then the normal sheaf $N_{f/X}$ has a torsion subsheaf of length $2$ at $p$.  Indeed, the curve defined by the equation $y^{2} + ax^{3} + bx^{2}y + cxy^{2} + dy^{3}$ has rational parametrization
\begin{equation*}
x = \frac{t^{2}}{a + bt + ct^{2} + dt^{3}} \qquad \qquad y = \frac{t^{3}}{a + bt + ct^{2} + dt^{3}}
\end{equation*}
and thus $dx$ is either $0$ or divisible by $t^{2}$ and $dy$ is divisible by $t^{2}$.

\begin{rema}
Note that an irreducible arithmetic genus $1$ curve $C$ in a smooth surface can only have nodes and cusps as singularities.  Indeed, a cohomological argument shows that the normalization of  $C$ must have genus $0$, $C$ can have at most one singular point $p$, and the preimage of $p$ under the normalization map must have length $2$.  Letting $\nu: \mathbb{P}^{1} \to C$ denote the normalization map, there is a three-dimensional subspace of $|\mathcal{O}(3)|$ which is constant on $\nu^{-1}(p)$.  This subspace defines a map $\mathbb{P}^{1} \to \mathbb{P}^{2}$ whose image is a cubic isomorphic to $C$.
\end{rema}

\subsection{Deformation theory of stable maps}
Fix an algebraically closed field $k$ and let $X$ be a smooth projective variety defined over $k$. We denote the Kontsevich moduli space of stable maps of genus $0$ by $\overline{M}_{0,0}(X)$. (See \cite{BF97}, \cite{B97}, and \cite{BM96} for the foundational theory of this coarse moduli space.)

Much of the theory of normal bundles to maps in characteristic $0$ goes through in characteristic $p$. We highlight here some useful previous results. Suppose that $C$ is a nodal arithmetic genus $0$ curve mapping to $X$ via a birational morphism $f$ which is a local immersion at each node of $C$. Under these hypotheses, the normal sheaf is defined as an extension
\[
0 \rightarrow \mathcal Ext_{\mathcal O_C}^1(Q, \mathcal O_C) \to  N_{f/X} \to \mathcal Hom_{\mathcal O_C}(K, \mathcal O_C) \to 0,
\]
where $K$ and $Q$ are the kernel and cokernel of $f^*\Omega_X^1 \to \Omega^1_C$. 
When $C$ is irreducible, the normal sheaf $N_{f/X}$ is simply the cokernel of $T_C \to f^*T_X$.
The space $H^0(C, N_{f/X})$ is the tangent space to the moduli space $\overline{M}_{0,0}(X)$ at the point corresponding to $f$ and $H^1(C, N_{f/X})$ is the obstruction space for the moduli space at $[f]$ (\cite{BF97}, \cite{B97}, and \cite{BM96}). In particular, the expected dimension of $\overline{M}_{0,0}(X)$ at $C$ is given by 
\[
h^0(C, N_{f/X}) - h^1(C, N_{f/X}) = -K_X.C + \dim X - 3,
\]
and this is always a lower bound for the dimension of an irreducible component containing $C$. 
It is natural to study $N_{f/X}$ by comparing it to the normal sheaves of the restriction of $f$ to the components of $C$. We have the following theorem which we use frequently:

\begin{theo}[\cite{GHS03} Lemma 2.6]
\label{thm-normalBundleNodalCurve}
Let $C$ be a nodal curve of arithmetic genus $0$ mapping to $X$ via a birational morphism $f$ that is a local immersion at every node. Let $g$ be the restriction of $f$ to a component $C_i$. Then sections of the normal sheaf $N_{f/X}$ restricted to a component $C_i$ are sections of $N_{g/X}$ with simple poles allowed at each node point in the direction of the other component. In particular, if $N_{g/X}$ is rank $1$ (i.e.~if $X$ is a surface) then $N_{f/X}|_{C_i}$ will simply be $N_{g/X}$ with the degree of the free part increased by the number of components meeting $C_i$.
\end{theo}
\begin{proof}
The problem is local so we may assume that $C$ is an embedded LCI curve.
Then the assertion follows from the discussion of \cite[the bottom of Page 1265]{HT08}.
\end{proof}

\begin{prop}
\label{prop-vanH1NodalCurve}
Let $E$ be a sheaf on a nodal curve $C$ of arithmetic genus $0$ satisfying the following two conditions:
\begin{itemize}
\item For each component $C_i$ in $C$, $H^1(C_{i}, E|_{C_i}) = 0$.
\item Every component $C_i$, except possibly one component $C_0$, satisfies that $E|_{C_i}$ is globally generated.
\end{itemize}
Then $H^1(C, E) = 0$.
\end{prop}
\begin{proof}
Recall the exact sequence
\[ 0 \to E \to \oplus_{i} E|_{C_i} \to \oplus_{j} E|_{p_j} \to 0. \]
The hypothesis that every component except $C_0$ satisfies that $E|_{C_i}$ is globally generated shows that the map $H^0(\sum_i E|_{C_i}) \to H^0(\sum_j E|_{p_j})$ is surjective. Thus, we see that $H^1(C,E)$ is isomorphic to $\oplus_i H^1(C_{i},E|_{C_i})$, which vanishes by hypothesis.
\end{proof}

\section{Low degree curves with higher than the expected dimension} \label{sect:expdimlowdegree}

Let $S$ be a weak del Pezzo surface.  Our first goal is to analyze when a family of rational curves of anticanonical degree $\leq 2$ has larger than the expected dimension.  This analysis will form the base case of an inductive argument which addresses curves of arbitrary anticanonical degree.

\begin{theo} \label{theo:expectdimlowdegree}
Let $S$ be a weak del Pezzo surface of degree $d$ over an algebraically closed field $k$.  When $d = 2$, we assume furthermore that $S$ is not the following exception:
\begin{enumerate}
\item $\ch(k)=2$ and $|-K_{S}|$ defines a purely inseparable generically finite morphism $f: S \to \mathbb{P}^{2}$.
\end{enumerate}
When $d = 1$, we assume furthermore that $S$ is not one of the following exceptions:
\begin{enumerate}
\setcounter{enumi}{1}
\item $\ch(k) = 2$ or $3$ and a general member of $|-K_S|$ is singular, or;
\item $\ch(k) = 2$ and $|-2K_S|$ defines a purely inseparable generically finite morphism $f : S \to Q$ where $Q$ is a quadric cone, or;
\item $\ch(k) = 2$ and $S$ admits a birational morphism to a surface as in (1) above.
\end{enumerate}
Let $M$ be a component of $\overline{M}_{0,0}(S)$ generically parametrizing a family of birational maps to curves $C$ with $-K_{S} \cdot C \leq 2$.  Then $M$ has the expected dimension unless $C$ is a $(-2)$-curve.
\end{theo}

We will prove Theorem \ref{theo:expectdimlowdegree} by analyzing each anticanonical degree $\leq 2$ separately.

\begin{lemm}
\label{lemm:(-2)-curve}
Let $S$ be a weak del Pezzo surface.  Every rational curve $C$ satisfying $-K_{S} \cdot C = 0$ is a smooth $(-2)$-curve on $S$.
\end{lemm}

\begin{proof}
By the Hodge Index Theorem we see that $C^{2} \leq 0$.  The arithmetic genus formula tells us that $C^{2} = 2p_a(C)-2$, so that $C^{2}$ is even and is at least $-2$.  If $C^{2}=0$ then by the Hodge Index Theorem $C$ must be proportional to $-K_{S}$, an impossibility.  Thus $C^{2}=-2$.  We deduce that $C$ has arithmetic genus $0$ and thus is smooth.
\end{proof}

\begin{lemm}
\label{lemm:lines}
Let $S$ be a weak del Pezzo surface of degree $d$.  Suppose that $M$ is a component of $\overline{M}_{0,0}(S)$ that parametrizes rational curves $C$ with $-K_{S} \cdot C = 1$.  Then either:
\begin{enumerate}
\item $C$ is a $(-1)$-curve.
\item $d=1$, $\dim(M) = 0$, and $C$ is a rational curve in $|-K_{S}|$.
\item $d=1$, $\dim(M) = 1$, and the curves parametrized by $M$ yield a quasielliptic fibration on the blow-up of $S$ along the basepoint of $|-K_{S}|$.  Furthermore in this case $S$ cannot be a del Pezzo surface.
\end{enumerate}
\end{lemm}

\begin{proof} The Hodge Index Theorem tells us that $dC^{2} - 1 \leq 0$.  The arithmetic genus formula tells us that $C^{2}-1 = 2p_a(C)-2$, so that $C^{2}$ is odd and is at least $-1$.  We deduce that the only options are:
\begin{itemize}
\item $C^{2} = -1$, $d$ arbitrary: in this case the arithmetic genus of $C$ is $0$, so $C$ is a $(-1)$-curve.
\item $C^{2} = 1$, $d=1$: in this case $C \in |-K_{S}|$ by the Hodge Index Theorem and $C$ has arithmetic genus $1$.  One possibility is that the general element of $|-K_{S}|$ is a smooth elliptic curve, in which case we are in (2).  The other option is that every element of $|-K_{S}|$ is a singular rational curve.  Note that $|-K_{S}|$ is not basepoint free, since any pencil of cubic curves in $\mathbb{P}^{2}$ will have nine base points.  Moreover since $(-K_{S})^{2} = 1$, two general members of $|-K_{S}|$ intersect at one point which is not a singular point of either curve.  Thus when we resolve the base locus of $|-K_{S}|$ by blowing up a single point the resulting fibration must be a quasi-elliptic fibration so that we are in (3). 
\end{itemize}
Finally, we note that families as in (3) do not exist on a del Pezzo surface of degree $1$.  It suffices to note that any quasi-elliptic pencil in the anticanonical system must contain a non-integral curve.  Indeed, by pushing these curves forward to $\mathbb{P}^{2}$ we obtain a family of rational curves through $9$ fixed points and Bend-and-Break guarantees that this family on $\mathbb{P}^{2}$ parametrizes a non-integral curve.
\end{proof}

\begin{lemm} \label{lemm:degree2curves}
Let $S$ be a weak del Pezzo surface of degree $d$.  Suppose that $M$ is a component of $\overline{M}_{0,0}(S)$ that generically parametrizes rational curves $f: \mathbb P^1 \to C \subset S$ with $-K_{S} \cdot C = 2$ and $f$ is birational.  Then either:
\begin{enumerate}
\item the component $M$ parametrizes the fibers of a conic fibration.  In this case $M$ has the expected dimension.
\item $d = 2$ and $M$ parametrizes curves in $|-K_{S}|$.  If $M$ has larger than the expected dimension then $|-K_{S}|$ does not define a separable map.
\item $d = 1$ and there is a birational map $\phi: S \to \widetilde{S}$ where $\widetilde{S}$ is a weak del Pezzo surface of degree $2$ such that $M$ parametrizes rational curves in $|-\phi^{*}K_{\widetilde{S}}|$.  If $M$ has larger than expected dimension then $|-\phi^{*}K_{\widetilde{S}}|$ does not define a separable map.
\item $d=1$ and $M$ parametrizes curves in $|-2K_{S}|$.  If $M$ has larger than expected dimension then either $|-K_{S}|$ defines a quasielliptic fibration or $|-2K_{S}|$ does not define a separable map.
\end{enumerate}
\end{lemm}

\begin{proof}
The Hodge Index Theorem tells us that $dC^{2} - 4 \leq 0$.  The arithmetic genus formula shows that $C^{2}-2 = 2g-2$, so that $C^{2}$ is even and is at least $0$.  We deduce that the only options are $C^{2} = 0,2,4$.

\textbf{Case 1:} $C^{2} = 0$, $d$ arbitrary: in this case the arithmetic genus of $C$ is $0$ and $H^{0}(S,\mathcal O(C)) = 2$.  Thus curves of this type are the fibers of a conic fibration.

\textbf{Case 2:} $C^{2} = 2$, $d=1,2$: in this case the arithmetic genus of $C$ is $1$ and $\deg(N_{f/S}) = 0$.  Suppose that $d=2$.  By the Hodge Index Theorem, $C$ is a member of $|-K_S|$. If $\Phi_{|-K_S|}$ is separable, then \cite[Proposition 4.4]{KN20} shows that a general element of $|-K_{S}|$ is smooth and we conclude that $M$ has the expected dimension. 

Next suppose that $d=1$.  In this case, we claim that $C$ is the pullback of a curve under a birational map to a degree $2$ weak del Pezzo surface.  To see this, it suffices to find a $(-1)$-curve which has vanishing intersection with $C$.  We claim that $K_{S}+C$ will be linearly equivalent to such a curve.  Indeed, we have
\begin{equation*}
(K_{S} + C) \cdot C = 0 \qquad (K_{S}+C) \cdot K_{S} = -1 \qquad  (K_{S} + C)^{2} = -1
\end{equation*}
Since $H^{2}(S,K_{S}+C)$ vanishes by Serre duality, Riemann-Roch shows that $H^{0}(S,K_{S}+C)$ is non-zero. This means that $K_S + C$ is linearly equivalent to an effective divisor and it follows from the above intersection numbers that it has the form of $E + D$ where $E$ is a $(-1)$-curve and $D$ is a non-negative linear combination of $(-2)$-curves. Then since $C$ is nef, we have $C \cdot D = E \cdot D + D^2 = 0$. We also have $-1 = (E+D)^2= -1 + 2E \cdot D + D^2$. Thus we conclude that $D^2 = 0$ proving that $D = 0$ by the Hodge Index Theorem.  Finally note that the deformations of $C$ on this degree $2$ weak del Pezzo yield a family satisfying (2).  Thus on our original surface we are in situation (3).

\textbf{Case 3:} $C^{2}=4$, $d = 1$: in this case the arithmetic genus of $C$ is $2$ and $C \in |-2K_{S}|$.  We will let $g: S \to Q$ denote the morphism to the quadric cone defined by $|-2K_{S}|$.  We let $S'$ denote the anticanonical model of $S$ and $g': S' \to Q$ the corresponding finite degree $2$ morphism.

To show that (4) holds, we must show that $M$ has the expected dimension if $|-K_{S}|$ does not define a quasielliptic fibration and $|-2K_{S}|$ defines a separable map.  From now on we assume both these conditions.  Let $B \subset Q$ denote the branch locus of $g: S \to Q$.  If $\ch(k) \geq 3$ then $B$ is the disjoint union of the cone vertex with a curve $B_1$ of degree $6$.  (If $B_{1}$ contained the cone vertex then one can show that $S'$ would have worse than canonical singularities, a contradiction.) If $\ch(k)=2$ then the cone vertex is contained in a dimension $1$ component of $B$ whose reduced part is a degree $3$ curve $B_{1}$.  

Every irreducible rational curve $C \in |-2K_{S}|$ will be singular, and thus the restriction $g|_{C}$ cannot realize $C$ as a simply branched cover of a smooth curve.  We conclude that $C$ must satisfy one of the following conditions:
\begin{enumerate}
\item $g(C)$ goes through a singular point of $B_{1}$.
\item $g(C)$ is a hyperplane section of $Q$ which is tangent to the divisor $B_{1}$ at a smooth point of $B_{1}$.
\end{enumerate}
We first show that there cannot be a $2$-dimensional family of rational curves as in (1) above.  In fact, we claim that there is not a $2$-dimensional sublocus of $M$ parametrizing rational curves through a fixed point $p \in S$.  Indeed, if there were such a family, then by applying Bend-and-Break we would obtain a $1$-dimensional family of rational curves with class $|-K_{S}|$.  But Lemma \ref{lemm:lines} would then contradict our assumption on $|-K_{S}|$.

We next rule out a $2$-dimensional family of rational curves as in (2) above.  If $C$ is a nodal rational curve then its normal sheaf is locally free and thus the deformations of $C$ have the expected dimension.  We conclude that if we have a $2$-dimensional family of deformations of $C$ then a general deformation must have a cusp. The rest of the argument depends on the characteristic.

\textbf{Case 3a:} First suppose that $\ch(k) \geq 3$.   If $C$ is a cuspidal rational curve in $|-2K_{S}|$ then its $g$-image is a hyperplane section of $Q \subset \mathbb{P}^{3}$ which has a point of tangency of order $\geq 3$ with $B_{1}$.  In other words, if $M$ fails to have the expected dimension then there is a two-dimensional family of planes meeting the curve $B_1$ to order at least three at some point. Since there is only a two-dimensional space of planes containing a tangent line of $B_1$, it follows that every tangent line to $B_1$ must be a flex. 
 Thus the tangent lines to $B_1$ meet $Q$ to order at least 3 at a given point, and in particular are all contained in $Q$.  In other words, the lines of the ruling are all tangent to $B_1$, and this implies that every member of $|-K_S|$ is singular. This contradicts with our assumption.
 %These lines all pass through a fixed point, so it follows that $B_1$ is a strange curve in the sense of Hartshorne \cite{hartshorne}, which is impossible by \cite[Chapter IV, Theorem 3.9]{hartshorne}. Thus, we see that when $\ch(k) \geq 3$ the space of rational curves in $|-2K_S|$ is at most 1-dimensional. 

\textbf{Case 3b:} Next suppose that $\ch(k) = 2$.  Since by assumption $g$ is separable, the anticanonical model $S'$ of $S$ is defined by an equation of the form $w^{2} + wf_{3} + f_{6}$ where $f_{3},f_{6}$ are homogeneous functions on $\mathbb{P}(1,1,2)$.
We will use coordinates $x_{0},x_{1},y$ on the weighted projective space.  First suppose that $f_{3}$ is irreducible and reduced.  Then by applying the automorphism group of $\mathbb{P}(1,1,2)$ we may suppose that
\begin{equation*}
f_{3} = x_{0}y - x_{1}^{3}.
\end{equation*}
Just as before, we rule out the possibility that every hyperplane section of the quadric cone that is tangent to the curve defined by $f_{3}=0$ has a cuspidal curve as a preimage.  The hyperplane sections of the quadric cone have equations of the form $a_{20}x_{0}^{2} + a_{11}x_{0}x_{1} + a_{02}x_{1}^{2} + cy$.  Since we are interested in the general tangent plane, without loss of generality we may suppose that $c=1$.  Eliminating $y$ and computing the discriminant, we see that a plane will be tangent to $f_{3}=0$ precisely when $a_{20} = a_{11}a_{02}$.  When this condition is met, the tangency point is $(x_0 : x_1 :  y) = (1:a_{11}^{1/2}:a_{11}^{3/2})$.

From now on we will work on the affine patch $x_{0} \neq 0$ isomorphic to $\mathbb{A}^{2}$.
Locally near the point $p = (a_{11}^{1/2},a_{11}^{3/2})$ the curve admits the rational parametrization
\begin{equation*}
t \mapsto \left( t + a_{11}^{1/2} , a_{02}t^{2} + a_{11}t + a_{11}^{3/2} \right)
\end{equation*}
sending $0 \mapsto p$.
Pulling back the defining equation for the double cover to this rational curve, we see that the preimage is defined by the equation
\begin{equation*}
w^{2} + w(t^{3} + (a_{11}^{1/2} + a_{02})t^{2}) + \widetilde{f_{6}}(t)
\end{equation*}
where the constant term of $\widetilde{f_{6}}$ is
\begin{align*}
b_{600} + a_{11}^{1/2}b_{510} + a_{11}b_{420} + a_{11}^{3/2}(b_{330} + b_{401}) + a_{11}^{2}(b_{240} + b_{311}) + a_{11}^{5/2}(b_{150} + b_{221}) + \\
a_{11}^{3}(b_{060} + b_{131} + b_{202}) + a_{11}^{7/2}(b_{041} + b_{112}) + a_{11}^{4}b_{022} + a_{11}^{9/2}b_{003}
\end{align*}
and the linear coefficient is 
\begin{align*}
b_{510} + a_{11}(b_{330} + b_{401}) + a_{11}^{2}(b_{150} + b_{221}) + a_{11}^{3}(b_{112} + b_{041}) + a_{11}^{4}b_{003}.
\end{align*}
Note that this double cover of $\mathbb{P}^{1}$ defines a cuspidal curve if and only if the constant and linear coefficients vanish.  If a general hyperplane section (i.e.~a general choice of $a_{11},a_{02}$) defines a cuspidal curve, we must have
\begin{align*}
b_{600} = b_{510} = \, & \, b_{420} = b_{022} = b_{003} = 0 \\
b_{330} = b_{401} \qquad b_{150} =  b_{221} & \qquad b_{112} = b_{041} \qquad b_{240} = b_{311} \\
b_{060} + b_{131} & + b_{202} = 0
\end{align*}
Equivalently, we have $f_{6} = (x_{0}y - x_{1}^{3}) g_{3}$ for some cubic equation $g_{3}$.  Thus the defining equation has the form
\[
w^2 + (x_0y - x_1^3)w + (x_0y - x_1^3)g.
\]
Then replacing $w$ by $w + g$, the equation becomes 
\[
w^2 + (x_0y - x_1^3)w + g^2.
\]
Replacing $w$ by $w + c(x_0y - x_1^3)$, we may assume that $g$ is a homogeneous polynomial in $x_0, x_1$. Then $(x_0:x_1:y:w) = (0:0:1:0)$ is a singular point of the surface.
On the patch where $y = 1$, locally analytically the equation looks like
\[
w^2 + (x_0 - x_1^3)w + g(x_0, x_1)^2
\]
in $\mathbb A^3/\mu_2$.

Now we will prove that the singularity at $(0,0,0)$ is worse than canonical.
First we claim that $\mathbb A^3/\mu_2$ has a terminal singularity.
Indeed, $\mathbb A^3/\mu_2$ is isomorphic to $$\Spec (k[x_0^2, x_1^2, w^2, x_0x_1, x_0w, x_1w]).$$
Let $\beta : W \to \mathbb A^3/\mu_2$ be the blow up of the origin and $E$ be the exceptional divisor. Then the discrepancy of $E$ is $1/2$.  
Let $\widetilde{S} \subset W$ be the strict transform of $S' \subset \mathbb A^3/\mu_2$.
Since the equation for $S'$ has no constant term and no odd degree monomial term, we can conclude that $S'$ is Cartier in $\mathbb A^3/\mu_2$ so that $\beta^*S' = \widetilde{S} + nE$ with a positive integer $n$.
Then the discrepancy of $E\cap \widetilde{S}$ in $\widetilde{S}$ over $S'$ is $\frac{1}{2} - n$ which is negative, proving the claim.
But the anticanonical model $S'$ must have only canonical singularities, and we conclude that any $S'$ of this type cannot admit a $2$-dimensional family of rational $-K_{X}$-conics.

\textbf{Case 3b':}  The other possibility is that $\ch(k) = 2$ and that $f_{3}$ is reducible or non-reduced.  Then by applying the automorphism group of $\mathbb{P}(1,1,2)$ we may suppose that either
\begin{equation*}
f_{3} = x_{0}y \qquad \qquad \textrm{or} \qquad \qquad f_{3} = g(x_{0},x_{1})
\end{equation*}
for some cubic $g$.
First suppose $f_{3} = x_{0}y$.
Writing as before $a_{20}x_{0}^{2} + a_{11}x_{0}x_{1} + a_{02}x_{1}^{2} + cy$ for the equation of a general hyperplane section, we see that the sections tangent to $f_{3} = 0$ will satisfy either $a_{02} = 0$, $c=0$, or $a_{11} = 0$.  The tangent point for every hyperplane section satisfying $a_{02} = 0$ is the point $(0:1:0)$; since we have already ruled out the case where each curve goes through the same point, we cannot get a $2$-dimensional family in the first case. In the second case the hyperplane section is not integral, thus we do not need to consider this case. It only remains to consider the case $a_{11}=0$.

Assuming that $a_{11} = 0$, then arguing as before we see that if every tangent to $f_{3}=0$ yields a cuspidal curve then $f_{6}$ is divisible by $y$.  We next need to check whether every member of this $2$-dimensional family of cuspidal curves on $S$ is rational.  Since $c \neq 0$, by rescaling we may suppose without loss of generality that $c=1$ so that $y = (\alpha x_{0} + \beta x_{1})^{2}$ where $\alpha = a_{20}^{1/2}$ and $\beta = a_{02}^{1/2}$.  We then eliminate the variable $y$ so that the equation of our singular curve in $\mathbb{P}_{x_{0},x_{1},w}(1,1,3)$ is given by
\begin{equation*}
w^{2} + x_{0}(\alpha x_{0} + \beta x_{1})^{2} w + (\alpha x_{0} + \beta x_{1})^{2} g_{4}(x_{0},x_{1})
\end{equation*}
for some degree $4$ equation $g$.  Since these curves have arithmetic genus $2$, if they are rational then they must have at least one other singularity besides the cusp.  (It is not possible for the equation above to define a worse-than-cuspidal singularity.)
The only option is that we have a singularity at $x_{0} = 0$, in which case $g_4$ must be divisible by $x_0^2$ for any $\alpha, \beta$.  This implies that
\[
b_{003} = b_{022} = b_{041} = b_{060} = b_{112} = b_{131} = b_{150} = 0.
\]
Hence we conclude that $f_6$ is divisible by $x_0^2y$.
 Thus the point $(0:0:1:0)$ in $\mathbb{P}(1,1,2,3)$ is contained in $S$ and this is a singularity worse than canonical by the argument above.

Suppose that $f_3$ is reduced but $f_{3}$ is a union of three lines.  Then the only hyperplane sections which are tangent to $f_{3} = 0$ will go through one of the singular points.  But we have already ruled out the case where each curve goes through the same point, so this situation cannot give a $2$-dimensional family of rational curves. 

When $f_3$ is non-reduced, we may assume that $f_3 = x_0^2x_1$ or $x_0^3$.
When $f_3 = x_0^2x_1$, a hyperplane section tangent to $f_3$ is given by
\[
a_{20}x_0^2 + a_{11}x_0x_1 + a_{02}x_1^2 = y
\]
such that the tangent point is given by $(0:1:a_{02})$.
If we can find a $2$-dimensional family of rational curves such that $a_{02}$ is fixed, then we can conclude a contradiction as before. So we may assume that $a_{02}$ is generic. Then a rational parametrization of the above section is given by
\[
(x_0:x_1:y) = (t:s:a_{20}t^2 + a_{11}ts + a_{02}s^2).
\]
For the resulting rational curve
\[
w^2 + t^2sw +f_6(t, s, a_{20}t^2 + a_{11}ts + a_{02}s^2)
\]
to have a cusp, by arguing as above we see that $f_6$ is divisible by $x_0^2$.  Thus the point $(0:0:1:0)$ is contained in $S$ and this is a singularity worse than canonical by the argument above.

Finally assume that $f_3 = x_0^3$. In this case by arguing as the case of $f_3 = x_0^2x_1$, we may conclude that $f_6$ is divisible by $x_0^2$. Thus one can deduce a contradiction as before.
\end{proof}

Altogether Lemma \ref{lemm:(-2)-curve}, Lemma \ref{lemm:lines}, and Lemma \ref{lemm:degree2curves} immediately imply Theorem \ref{theo:expectdimlowdegree}.

\subsection{Pathological weak del Pezzo surfaces}
\label{subsec:examples}

We next classify the weak del Pezzo surfaces $S$ which admit a dominant family of rational curves of low degree which has larger than the expected dimension.  In view of later applications, we will split these surfaces into $3$ different types.  We emphasize that these three types are not mutually exclusive.

Our description will be based upon \cite{KN20} and \cite{KN20b} which  classify the weak del Pezzo surfaces in characteristics $2$ and $3$ such that the anticanonical model has Picard rank $1$.  \cite[Table 1]{KN20b} gives a complete list of such surfaces based on the type of singularities of the anticanonical model; we will refer to this table in our discussion.

\subsubsection{Type 1} The first type of pathological weak del Pezzo surface will satisfy the following equivalent conditions:

\begin{clai} \label{clai:type1}
Let $S$ be a weak del Pezzo surface.  Then the following are equivalent:
\begin{enumerate}
\item $S$ admits a dominant family of rational curves of anticanonical degree $1$.
\item $S$ is a weak del Pezzo surface of degree $1$ such that every element of $|-K_{S}|$ is singular.
\item $S$ is a weak del Pezzo surface of degree $1$ and if we blow-up the basepoint of $|-K_{S}|$ we obtain a quasi-elliptic fibration.
\end{enumerate}
In particular such surfaces can only occur in characteristic $2$ or $3$.
\end{clai}

\begin{proof}
The equivalence of (1) and (2) is an immediate consequence of Lemma \ref{lemm:lines}.  To prove the equivalence of (2) and (3), we just need to note that if every element of $|-K_{S}|$ is singular then it is not possible for all the singularities to coincide (since the intersection number of two curves in the family should be $1$).
\end{proof}

Over an algebraically closed field of characteristic $2$ or $3$, such surfaces have been classified by \cite[Theorem 1.4]{KN20}.  In characteristic $2$ there are $3$ infinite families and $4$ other surfaces whose singularity types are:
\begin{equation*}
E_{8}, D_{8}, A_{1}+E_{7},  2A_{1}+D_{6}, 2D_{4}, 4A_{1}+D_{4}, 8A_{1}
\end{equation*}
In characteristic $3$ there are three surfaces whose singularity types are:
\begin{equation*}
E_{8}, A_{2}+E_{6}, 4A_{2}
\end{equation*}

\subsubsection{Type 2} The second type of pathological weak del Pezzo surface will satisfy the following equivalent conditions:

\begin{clai} \label{clai:type2}
Let $S$ be a weak del Pezzo surface.  Then the following are equivalent:
\begin{enumerate}
\item $S$ has degree $2$ and admits a family of rational curves of anticanonical degree $2$ with larger than the expected dimension.
\item $S$ has degree $2$ and every element of $|-K_{S}|$ is singular.
\item $|-K_{S}|$ defines a purely inseparable morphism $g: S \to \mathbb{P}^{2}$.
\item The anticanonical model $S'$ of $S$ admits a purely inseparable map $f: \mathbb{P}^{2} \to S'$ of degree $2$ such that $f^{*}(-K_{S}') \cong \mathcal{O}(2)$.
\end{enumerate}
In particular such surfaces can only occur in characteristic $2$.
\end{clai}

\begin{proof}
Lemma \ref{lemm:degree2curves} shows that the only possible family of conics with larger than expected dimension on a weak del Pezzo surface of degree $2$ must lie in $|-K_{S}|$.  Since this linear series has dimension $2$, we see that (1) implies (2). Lemma \ref{lemm:degree2curves} shows that (2) implies (3).

We next show that (3) implies (4).  Let $S'$ denote the anticanonical model of $S$.  (3) asserts that $|-K_{S}|$ defines a purely inseparable morphism $g: S \to \mathbb{P}^{2}$.  This necessarily implies that $S$ has degree $2$ and that $g$ has degree $2$.  The Stein factorization of $g$ will be the anticanonical model $S'$ of $S$; in other words, $S'$ will be the normalization of $\mathbb{P}^{2}$ inside the function field of $S$.  Since $K(S)$ is obtained by adjoining a single square root to $K(\mathbb{P}^{2})$, we see that the Frobenius morphism $\mathbb{P}^{2} \to \mathbb{P}^{2}$ factors through $g$.  In this way we obtain a purely inseparable degree $2$ map $f: \mathbb{P}^{2} \to S'$.  Furthermore we have
\begin{equation*}
f^{*}(-K_{S'}) = f^{*}g^{*}\mathcal{O}(1) = \mathcal{O}(2).
\end{equation*}

Finally, we show that (4) implies (1).  By taking intersection numbers it is clear that $S$ has degree $2$.  The lines on $\mathbb{P}^{2}$ will map to a two-dimensional family of curves on $S'$ of anticanonical degree $\leq 2$ with no basepoints.  Since there is no such family on $S'$ of anticanonical degree $1$ by Lemma \ref{lemm:lines}, we see that the restriction of $f$ to each line is birational.  Thus $S$ admits a two-dimensional family of rational curves of anticanonical degree $2$.
\end{proof}

Over an algebraically closed field of characteristic $2$, the surfaces satisfying (2) have been classified by \cite[Theorem 1.4]{KN20}.  There are exactly $4$ such surfaces, whose singularity types are:
\begin{equation*}
E_{7}, A_{1}+D_{6},  3A_{1} + D_{4}, 7A_{1}
\end{equation*}

\subsubsection{Type 3}  The third type of pathological weak del Pezzo surface will satisfy the following equivalent conditions:

\begin{clai} \label{clai:type3}
Let $S$ be a weak del Pezzo surface.  Then the following are equivalent:
\begin{enumerate}
\item $S$ has degree $1$ and every element of $|-2K_{S}|$ is a singular rational curve.
\item $S$ has degree $1$ and the morphism to the quadric cone $g: S \to Q$ defined by $|-2K_{S}|$ is purely inseparable.
\item The anticanonical model $S'$ of $S$ admits a purely inseparable map $f: Q \to S'$ of degree $2$ from the quadric cone $Q$ such that $f^{*}(-K_{S}') \cong \mathcal{O}(1)$.
\end{enumerate}
In particular such surfaces can only occur in characteristic $2$.
\end{clai}

\begin{proof}
We first show that (1) implies (2).  Suppose for a contradiction that the morphism $g: S \to Q$ is separable.  Note that $g$ is finite on the complement of the $(-2)$-curves in $S$.  Let $U \subset S$ be the open set which is the complement of the $(-2)$-curves and the preimage of the singular locus of the branch divisor.  Then $U$ admits a decomposition into locally closed subsets $L_{1},L_{2}$ where $L_{1}$ is the ramification divisor and $L_{2}$ is its complement.  By construction $L_{1},L_{2}$ are smooth and the restriction of $g$ to both $L_{1}$ and $L_{2}$ is a smooth morphism.  By \cite[Corollary 4.6]{Spreafico} we conclude that the pullback of a general hyperplane in $Q$ to $U$ will be smooth.  Since we only removed the $(-2)$-curves and a codimension $2$ set, the general pullback of a hyperplane in $U$ is also projective.  Altogether we see that a general element in $|-2K_{S}|$ is smooth, contradicting (1).

The proof that (2) implies (3) follows from a Frobenius factoring argument as in Claim \ref{clai:type2}.  

Finally, to see that (3) implies (1) we note that $Q$ admits a $3$-dimensional family of rational curves of degree $2$.  These map to a $3$-dimensional family of rational curves on $S'$ of anticanonical degree $\leq 2$ with no basepoints.  Since by Lemma \ref{lemm:lines} there is no such family on $S'$ of anticanonical degree $1$, we see that the restriction of $f$ to each rational curve is birational.  Lemma \ref{lemm:degree2curves} shows that a three-dimensional family of rational curves of anticanonical degree $2$ must lie in $|-2K_{S}|$ and we conclude that all the curves parametrized by $|-2K_{S}|$ are singular.
\end{proof}

It turns out that Type 3 surfaces are exactly the same as the Type 1 surfaces which have characteristic $2$.  We will demonstrate this in Proposition \ref{prop:type3istype1} after developing the theory of $a$-covers.

For now, it will suffice to show that every Type 3 surface also has Type 1.  Suppose that $Y$ is a Type 3 surface so that its anticanonical model $S'$ admits a purely inseparable degree $2$ map $f: Q \to S'$ from the quadric cone.   Each line on $Q$ maps birationally onto a $-K_{S}$-line in $S'$.  In particular by pulling back to $S$ we see there must be a one-dimensional family of $-K_{S}$-lines.  Thus Type 3 surfaces are a subclass of Type 1 surfaces.

Combining this classification with Theorem \ref{theo:expectdimlowdegree}, we obtain:

\begin{coro} \label{coro:lowdegreeexpectdimanticanonical}
Let $S$ be a weak del Pezzo surface.  If $S$ carries a dominant family of rational curves of anticanonical degree $\leq 2$ with larger than the expected dimension then $S$ has Type 1, Type 2, or Type 3.  In particular every curve parametrized by $|-K_{S}|$ is singular.
\end{coro}

\begin{proof}
Let us start with the first claim.  According to Lemma \ref{lemm:lines} and Lemma \ref{lemm:degree2curves} the only case which needs consideration is when $S$ is a weak del Pezzo surface with degree $1$ that admits a birational map $\phi: S \to T$ to a surface of Type 2 that contracts a $(-1)$-curve.  In this case every integral member of $|-K_{T}|$ is a rational curve.  Thus $S$ also has this same property, and in particular must have Type $1$.

Note that both Type 1 and Type 2 surfaces have the property that every element of $|-K_{S}|$ is singular.  Furthermore, every Type 3 surface also has Type 1 and thus has the same property.  This proves the second claim.
\end{proof}

\section{Low degree curves and inseparable families} \label{sect:lowdegree}

In this section first we focus our attention on del Pezzo surfaces $S$ of degree $\geq 2$.  Our goal is to classify the inseparable families of rational curves $C$ on $S$ which satisfy $-K_{S} \cdot C \leq 3$.

\begin{theo} \label{theo:separablefamlowdegree}
Let $S$ be a del Pezzo surface of degree $\geq 2$.  Let $M$ be a component of $\overline{M}_{0,0}(S)$ generically parametrizing a dominant family of curves $C$ with $-K_{S} \cdot C \leq 3$. When $d = 3$, we assume furthermore that $S$ is not the following exception:
\begin{enumerate}
\item $\ch(k)=2$ and $S$ is the Fermat cubic surface $w^{3}+x^{3}+y^{3}+z^{3}=0$.
\end{enumerate}
When $d = 2$, we assume furthermore that $S$ is not one of the following list of exceptions:
\begin{enumerate}
\setcounter{enumi}{1}
\item $\ch(k)=3$ and $S$ is the double cover of $\mathbb{P}^{2}$ ramified along the curve $zx^{3} + xy^{3} + yz^{3}$.
\item $\ch(k)=2$ and $S$ is a double cover of $\mathbb{P}^{2}$ defined by the equation $w^{2} + wy^2 + g_{4}$ where $g_4$ is a homogeneous polynomial in $x, y, z$. 
%\item $\ch(k)=2$ and $S$ is the blow-up of the Fermat cubic surface.
\end{enumerate}
Then $M$ parametrizes a separable family of curves
\end{theo}

To prove Theorem \ref{theo:separablefamlowdegree}, first note that by Lemma \ref{lemm:(-2)-curve} and Lemma \ref{lemm:lines} we only need to consider dominant families of rational curves of anticanonical degrees $2$ and $3$.  We will analyze each case separately.

\begin{lemm} \label{lemm:degree2curvessep}
Let $S$ be a del Pezzo surface of degree $d \geq 2$.  Suppose that $M$ is a component of $\overline{M}_{0,0}(S)$ that parametrizes a dominant family of rational curves $C$ with $-K_{S} \cdot C = 2$.  Then either:
\begin{enumerate}
\item the component $M$ parametrizes the fibers of a conic fibration and the general fiber is a free rational curve.
\item $d = 2$, $\dim(M) = 1$, and either
\begin{enumerate}
\item $M$ parametrizes a separable family of rational curves in $|-K_{S}|$, or
\item $\ch(k)=3$, $S$ is the double cover of $\mathbb{P}^{2}$ ramified along the curve $zx^{3} + xy^{3} + yz^{3}$, and $M$ is the dual curve of this quartic parametrizing an inseparable family of rational curves in $|-K_S|$, or
\item $\ch(k)=2$, $S$ is a double cover of $\mathbb{P}^{2}$ defined by the equation $w^{2} + wg_{2} + g_{4}$ where $g_{2} = \ell^{2}$ is a double line, and $M$ parametrizes the preimages of a $1$-dimensional family of lines in $\mathbb{P}^{2}$ where for any point $p \in V(\ell)$ the slope of the line through $p$ is determined by the derivatives of $g_{4}$ at $p$. This family is inseparable.
\end{enumerate}
\end{enumerate}
\end{lemm}

\begin{proof}
Since $M$ parametrizes a dominant family, a general member $f: \mathbb P^1 \to C\subset S$ parameterized by $M$ cannot be a multiple cover of a line. Thus $f : \mathbb P^1 \to C$ is birational.

As in Lemma \ref{lemm:degree2curves} the Hodge Index Theorem implies that $C^{2} = 0,2,4$.  Furthermore, as explained in Lemma \ref{lemm:degree2curves} in the $C^{2}=0$ case the curves are fibers of a conic fibration. Since a general fiber is smooth, the normal sheaf is locally free which implies that such a rational curve is free.  The $C^{2} = 4$ case can only occur when $d=1$ and thus is not a concern for us.

The only remaining case to consider is when $C^{2} = 2$ and $d=2$.  In this case the curve $C$ lies in $|-K_{S}|$; the arithmetic genus of $C$ is $1$ and $\deg(N_{f/S}) = 0$.  We will argue separately the cases when $\ch(k)>2$ and when $\ch(k)=2$.

\textbf{Case 1:} First suppose that the ground field $k$ has characteristic $p \geq 3$.  The anticanonical linear series defines a double cover $f: S \to \mathbb{P}^{2}$ that is ramified along a smooth quartic curve $B$. The $f$-images of the curves parametrized by $M$ will be the lines in $\mathbb{P}^{2}$ which are tangent to $B$ and thus $M$ will be the dual curve to $B$. This family of lines will define a non-separable cover if and only if the general curve has normal bundle $\mathcal{O}(-1) \oplus k(p)$, and hence is a cuspidal curve. This occurs precisely when the line is flex to the branch curve $B$. By \cite[(4.5)]{Hefez89}, this is equivalent to the curve being non-reflexive (which means that the Gauss map to the curve $B$ is purely inseparable).

\cite{Pardini86} analyzes the smooth plane curves $B$ in characteristic $p \geq 3$ which are non-reflexive.  \cite[Corollary 2.2]{Pardini86} shows that if $B$ is smooth and non-reflexive of degree $4$ then $\ch(k) = 3$.  \cite[Proposition 3.7]{Pardini86} shows that every smooth reflexive curve of degree $4$ is projectively equivalent to $zx^{3} + xy^{3} + yz^{3}$.

\textbf{Case 2:} Next suppose that the ground field $k$ has characteristic $2$.  In this case it is still true that the anticanonical linear series defines a separable double cover $f: S \to \mathbb{P}^{2}$ whose branch divisor $B$ is a double conic.  (However, it is possible for this conic to be singular.)  The equation for $S$ has the form $w^{2} + wg_{2} + g_{4}$ where $g_{2},g_{4}$ are homogeneous functions of degree $2,4$ respectively and $B$ is defined by $g_{2}^2$.

Suppose we take a line $\ell \subset \mathbb{P}^{2}$ not contained in $B$ and set $C = f^{-1}\ell$.  We wish to know when $C$ is cuspidal.
Since the map $C \to {\ell}$ is generically \'etale, $C$ can only be singular along $\ell \cap B$.  We claim that $C$ can only have a cusp when it is tangent to the conic defined by $g_{2} = 0$.  Indeed, suppose that $p \in B \cap \ell$ and change coordinates via $w \mapsto w-\alpha$ so that $w$ vanishes at $p$.  (Note that this coordinate transformation might change $g_4$ but will not change $g_2$.)  Let $x$ be a local coordinate for $p$ on $\ell$.  If $g_2|_{\ell}$ is not $x^2$, then locally $g_2$ can be taken to be proportional to $x$, but then the curve $w^2+wx+g_4$ will have at worst a nodal singularity.  Thus if the preimage of $\ell$ is cuspidal then $\ell$ must be tangent to the conic defined by $g_{2} = 0$ along the image of the cusp.

We now separate into three further subcases:

\textbf{Case 2a:} 
$g_{2}$ defines a smooth conic.  After a coordinate change we may suppose that the equation for $S$ has the form $w^{2} + w(yz - x^{2}) + g_{4}$.  We will write $g_{4} = \sum_{i+j+k=4} a_{ijk} x^{i}y^{j}z^{k}$.

The family of tangent lines to $yz - x^{2}=0$ have equations of the form $b_{1}y+b_{2}z=0$.  Since we are interested in what happens to a general line in this $1$-parameter family, we may for simplicity assume that the equation of the line is $z=by$.  Then the restriction of the equation defining $S$ to the line $\ell$ can be written as
\begin{align*}
w^{2} & + w(by^{2} - x^{2}) + a_{400} x^{4} + (a_{310} + a_{301}b) x^{3}y + (a_{220} + a_{211}b + a_{202}b^{2}) x^{2}y^{2} \\
& + (a_{130} + a_{121}b + a_{112}b^{2} + a_{103}b^{3})xy^{3} + (a_{040}+a_{031}b + a_{022}b^{2} + a_{013} b^{3} + a_{004}b^{4})y^{4}
\end{align*}
The tangency point $p$ has coordinates $(\sqrt{b}:1)$ on the line and over this point
\begin{align*}
w^{2} = & \, a_{004} b^{4} + a_{103} b^{3} \sqrt{b} + (a_{202} + a_{013}) b^{3} +  (a_{301}+a_{112}) b^{2}\sqrt{b} + (a_{400}+a_{211} + a_{022}) b^{2} \\ & +  (a_{310} + a_{121}) b \sqrt{b} + (a_{220} + a_{031}) b + a_{130}\sqrt{b} + a_{040}.
\end{align*}
Rewriting the equation around the point $(\sqrt{b}:1)$ and substituting $w' = w - \alpha$ where $\alpha$ is the square root of the righthand side of the previous equation, we obtain
\begin{align*}
w'^{2}  & + w'(x-\sqrt{b})^{2}  + a_{400} (x-\sqrt{b})^{4} + (a_{310} + a_{301}b) (x-\sqrt{b})^{3} \\
& + (a_{220} + a_{310}\sqrt{b}  + a_{211}b + a_{301}b\sqrt{b}  + a_{202}b^{2}) (x-\sqrt{b})^{2} + \alpha(x-\sqrt{b})^{2} \\
& + (a_{130} + (a_{121}+a_{310}) b+ (a_{112}+a_{301})b^{2} + a_{103}b^{3})(x-\sqrt{b}) %\\
\end{align*}
For this equation to define a cusp, the coefficients of $(x-\sqrt{b})$ must vanish.  Furthermore, we can only have a purely inseparable family when each of these lines yields a cusp regardless of the value of $b$.  This forces the coefficients to be identically zero:
\begin{align*}
a_{130} = a_{103}  = 0  \qquad a_{121} = a_{310}  \qquad a_{112} = a_{301}
\end{align*}
However, any surface $S$ defined by an equation whose coefficients satisfy these conditions will be singular.  Indeed, consider the chart where $z=1$.  On this chart we have
\begin{align*}
d/dw & = y-x^{2} \\
d/dx & = a_{310}y(x^{2}-y) + a_{301}(x^{2}-y) \\
d/dy & = w + a_{310}x^{3} + a_{211}x^{2} + a_{031}y^{2} + a_{112}x + a_{013}
\end{align*}
and we are looking for points on $S$ where all three equations simultaneously vanish.  Note that the vanishing of $d/dw$ implies the vanishing of $d/dx$.  Thus $S$ will always have a singular point; indeed, isolating $y$ in the equation for $d/dw$ and $w$ in the equation for $d/dy$ and substituting into the equation for $S$ we obtain a polynomial in $x$ whose roots will correspond to singular points of $S$.

\textbf{Case 2b:} $g_{2}$ defines a reducible conic.  After a coordinate change we may suppose that the equation for $S$ has the form $w^{2} + w(yz) + g_{4}$.  We will write $g_{4} = \sum_{i+j+k=4} a_{ijk} x^{i}y^{j}z^{k}$.

The family of tangent lines to $yz=0$ have equations of the form $by+cz=0$.  Since we are interested in what happens to a general line in this $1$-parameter family, we may for simplicity assume that the equation of the line is $by = z$.  Then the restriction of the equation defining $S$ to the line $\ell$ can be written as
\begin{align*}
w^{2} & + w(by^{2}) + a_{400} x^{4} + (a_{310} + a_{301}b) x^{3}y + (a_{220} + a_{211}b + a_{202}b^{2}) x^{2}y^{2} \\
& + (a_{130} + a_{121}b + a_{112}b^{2} + a_{103}b^{3})xy^{3} + (a_{040}+a_{031}b + a_{022}b^{2} + a_{013} b^{3} + a_{004}b^{4})y^{4}
\end{align*}
We are interested in the behavior over the point $(1:0)$ so that $w = \sqrt{a_{400}}$.  Rewriting the equation so that it is centered around this point, we see that every line will yield a cusp precisely when $a_{310} = a_{301} = 0$.  However, these conditions force $S$ to be singular over the point $(1:0:0)$.

\textbf{Case 2c:}   $g_{2}$ defines a non-reduced conic.  After a coordinate change we may suppose that the equation for $S$ has the form $w^{2} + wy^{2} + g_{4}$.  We will write $g_{4} = \sum_{i+j+k=4} a_{ijk} x^{i}y^{j}z^{k}$.

Note that every line is tangent to the curve $y^{2}=0$.  We would like to understand the situation when there is a $1$-parameter family of lines whose preimages are cuspidal.  Since the equation of $S$ is symmetric in $x$ and $z$, without loss of generality we may assume that the general line in our family has an equation of the form $z = bx + cy$.  Then the restriction of the equation defining $S$ to the line $\ell$ can be written as
\begin{align*}
w^{2} & + wy^{2} + (a_{400}+a_{301}b + a_{202}b^{2} + a_{103}b^{3} + a_{004}b^{4}) x^{4} \\
& + (a_{310} + a_{301}c + a_{211}b + a_{112}b^{2} + a_{013}b^{3} + a_{103}b^{2}c ) x^{3}y \\
& + (a_{220} + a_{211}c + a_{121}b + a_{202}c^{2} + a_{022}b^{2} + a_{103}bc^{2} + a_{013}b^{2}c) x^{2}y^{2} \\
& + (a_{130} + a_{121}c + a_{031}b + a_{112}c^{2} + a_{103}c^{3} + a_{013}bc^{2})xy^{3} \\
& + (a_{040}+a_{031}c + a_{022}c^{2} + a_{013}c^{3} + a_{004}c^{4})y^{4}
\end{align*}
Arguing as in the other cases, this equation will define a cusp precisely when the coefficient of $x^{3}y$ vanishes.  Furthermore, we want to ensure that there is a $1$-dimensional family of lines for which these coefficients vanish.  This gives the condition
\begin{equation*}
c(a_{301} + a_{103}b^{2}) = a_{310} + a_{211}b + a_{112}b^{2} + a_{013}b^{3}
\end{equation*}
for our $1$-dimensional family.  Note that the intersection of the line $z = bx + cy$ with the line $y=0$ is given by the point $p = (1:0:b)$.  Then it is easy to check that
\begin{equation*}
c = \frac{dg_{4}/dy(p)}{dg_{4}/dx(p)}.
\end{equation*}
\end{proof}

Next we turn to rational curves of anticanonical degree $3$.

\begin{lemm} \label{lemm:wdphighchar}
Let $S$ be a del Pezzo surface of degree $d \geq 2$.  Suppose that $M$ is a component of $\overline{M}_{0,0}(S)$ that parametrizes a dominant family of rational curves $C$ with $-K_{S} \cdot C = 3$.  Then either:
\begin{enumerate}
\item the component $M$ defines a separable family of stable maps which are generically birational maps to smooth free curves.
\item $d = 3$, $\dim(M) = 2$, and either
\begin{enumerate}
\item $M$ parametrizes a separable family of rational curves in $|-K_{S}|$, or
\item $\ch(k)=2$, $S$ is the Fermat cubic surface in $\mathbb{P}^{3}$, and $M$ is the dual variety parametrizing hyperplanes tangent to $S$. In this case, the family is inseparable.
\end{enumerate}
\item $d=2$, $\dim(M)=2$ and either
\begin{enumerate}
\item $M$ parametrizes a separable family of rational curves, or
\item $\ch(k)=2$, $S$ is the blow-up of the Fermat cubic surface $S'$ in $\mathbb{P}^{3}$, and $M$ is birational to the dual variety of $S'$ and parametrizes the pullbacks of rational members of $|-K_{S'}|$.  In this case, the family is inseparable.
\end{enumerate}
\end{enumerate}
\end{lemm}

\begin{proof}
Since $M$ parametrizes a dominant family, a general member $f: \mathbb P^1 \to C\subset S$ parameterized by $M$ cannot be a multiple cover of a line. Thus $f : \mathbb P^1 \to C$ is birational.

The Hodge Index Theorem tells us that $dC^{2} - 9 \leq 0$.  The arithmetic genus formula tells us that $C^{2}-3 = 2g-2$, so that $C^{2}$ is odd and is at least $1$.  Since by assumption $d \geq 2$, we also have that $C^{2} \leq 3$.  We deduce that the only options are:

\textbf{Case 1:} $C^{2} = 1$, $d$ arbitrary: in this case the arithmetic genus of $C$ is $0$, so $C$ is free.  In fact, $C$ must be the strict transform of a general line under a birational map to $\mathbb{P}^{2}$.  Such families are separable and have the expected dimension.

\textbf{Case 2:} $C^{2} = 3$, $d=2,3$: In this case the arithmetic genus of $C$ is $1$ and $\deg(N_{f/S}) = 1$.  Since $C$ is a deformation of an elliptic curve, this implies that either $N_{f/S} = \mathcal{O}(1)$, $N_{f/S} = \mathcal{O} \oplus k(p)$, or (in characteristic $2$ only) $N_{f/S} = \mathcal{O}(-1) \oplus k[t]/(t^2)$ where the torsion part is supported on $p$.  If $k$ has characteristic $> 2$ then the normal sheaf of a general stable map is globally generated and thus $M$ is a separable family.

We also need to address separability when $\ch(k)=2$,  $d=2$ or $3$, and $S$ is a del Pezzo surface.  First suppose that $d=3$.  In this case $S$ is a smooth cubic hypersurface defined by an equation $f = \sum y_{ijkl} x_{0}^{i}x_{1}^{j}x_{2}^{k}x_{3}^{l}$.  By Section~\ref{subsec:node_cusp} the family of rational curves obtained by singular hyperplane sections of $S$ will be inseparable if and only if the general such curve is cuspidal.

Suppose $P$ is the plane defined by the equation $\sum_{i=0}^{3} z_{i}x_{i} = 0$.  Using standard facts about elliptic curves (see for example \cite{Tate74}), we see that intersection $P \cap S$ will be singular if and only if the discriminant $\Delta$ of $f|_{P}$ vanishes and in this case it will be cuspidal if and only if $c_{4}$ vanishes.  A computation shows that
\begin{equation*}
c_{4} = y_{0111}^{4}z_{0}^{4} + y_{1011}^{4}z_{1}^{4} + y_{1101}^{4}z_{2}^{4} + y_{1110}^{4}z_{3}^{4} = ( y_{0111}z_{0} + y_{1011}z_{1} + y_{1101}z_{2} + y_{1110}z_{3} )^{4}.
\end{equation*}
If every singular hyperplane section of $S$ is cuspidal, then the dual variety must contain the plane defined by $c_{4}=0$ or $c_{4}$ must be identically zero.  However, the first option is impossible due to the classification of ``strange hypersurfaces'' in \cite[Theorem 7]{KP91}.  Thus the cubic surfaces $S$ for which every singular hyperplane section is cuspidal are exactly those which satisfy
\begin{equation*}
y_{0111} = y_{1011} = y_{1101} = y_{1110} = 0.
\end{equation*}
We then claim that every smooth cubic satisfying this condition is projectively equivalent to the Fermat cubic surface.  Note that the locus of cubic surfaces satisfying this condition is invariant under the action of 
$PGL_{4}(k)$. For any cubic surface $S$, $\Aut(S)$ injects into the group of automorphisms of configurations of $(-1)$-curves on $S$, so $\Aut(S)$ is a finite group. (See \cite{DD19} for this claim.) Therefore the orbit of $S$ under $PGL_{4}(k)$ will be $15$-dimensional. Since the projective space of cubic surfaces satisfying the above condition is also $15$-dimensional, the claim follows.  

Next suppose that $d=2$.  We claim that in this case $C$ is the pullback of a curve under a birational map to a degree $3$ del Pezzo surface.  To see this, it suffices to find a $(-1)$-curve which has vanishing intersection with $C$.  We claim that $K_{S}+C$ will be linearly equivalent to such a curve.  Indeed, we have
\begin{equation*}
(K_{S} + C) \cdot C = 0 \qquad (K_{S}+C) \cdot K_{S} = -1 \qquad  (K_{S} + C)^{2} = -1
\end{equation*}
Since $H^{2}(S,K_{S}+C)$ vanishes by Serre duality, Riemann-Roch shows that $H^{0}(S,K_{S}+C)$ is non-zero, finishing the argument.  We conclude that the family of rational curves containing $C$ is separable unless $S$ is the blow-up of the Fermat cubic surface.
\end{proof}

\begin{proof}[Proof of Theorem \ref{theo:separablefamlowdegree}:]
Combining Lemma \ref{lemm:degree2curvessep} and Lemma \ref{lemm:wdphighchar}, we only need to show that if $S$ is a degree $2$ del Pezzo surface in characteristic $2$ obtained by blowing-up the Fermat cubic surface then the ramification locus of the map $S \to \mathbb{P}^{2}$ defined by $|-K_{S}|$ is a double line.  Note that every smooth hyperplane section of the Fermat cubic surface is a supersingular elliptic curve (see \cite[Theorem 1.1]{Homma97}).  Thus every smooth curve in $|-K_{S}|$ also has the same property.  Then \cite[Theorem 0.3]{Saito17} proves the desired property of the anticanonical linear series of $S$.
\end{proof}

We next show that there are no issues with separability when the characteristic is sufficiently large. Recall from (\ref{delta}) that the function $\delta(d)$ is defined by:
\begin{equation*}
\delta(d) = 
\begin{cases}
2 & \text{ if $d \geq 4$}\\
3 & \text{ if $d = 2, 3$}\\
11 & \text{ if $d = 1$}.
\end{cases}
\end{equation*}

\begin{theo} \label{theo:separabilitylargechar}
Let $k$ be an algebraically closed field of characteristic $p$. Let $S$ be a del Pezzo surface of degree $d$ over $k$. Assume that $p \geq \delta(d)$. When $d = 2$ and $p = 3$, we further assume that $S$ is not isomorphic to the del Pezzo surface listed in Theorem~\ref{theo:separablefamlowdegree}.(2).

Then any dominant family of rational curves on $S$ of anticanonical degree $\leq 3$ contains a free rational curve. In particular, any dominant component parametrizing rational curves of anticanonical degree $\leq 3$ is separable so that it has expected dimension.
\end{theo}

\begin{proof}
Since the surfaces under consideration do not admit a dominant family of $-K_{S}$-lines, we see that a general map in a dominant family of maps of degree $\leq 3$ must be birational onto its image. 

When $d\geq 3$, every $-K_S$-conic is smooth and every $-K_S$-cubic has at most $1$ cusp. Thus it follows from a normal bundle calculation that such curves are free.  When $d = 2$ and $p = 3$, $-K_S$-conics are handled by Lemma~\ref{lemm:degree2curvessep} and every $-K_S$-cubic has at most one cusp and thus is free. So we only need to consider when $d = 1$ and $p \geq 11$ or $d = 2$ and $p \geq 5$. Let us discuss the case of $d = 1$ as the other case is similar.

Let $M$ be any component of $\overline{M}_{0,0}(S)$ which parametrizes a dominant family of stable maps of anticanonical degree $\leq 3$ such that the general map has irreducible domain and is birational onto its image in $S$.  Let $T$ be the normalization of a curve in $M$ which parametrizes a dominant family of irreducible curves.  It suffices to show that the restriction of the tangent bundle of $S$ to a general curve parametrized by $T$ is globally generated.

Let $\mathcal{C}$ denote the normalization of the one-pointed family over $T$ equipped with the evaluation map $ev: \mathcal{C} \to S$.  By \cite[Lemma 6.1]{IIL20} we have a diagram
\begin{equation*}
\xymatrix{ \mathcal{C} \ar[d]_{s} \ar[r]^{g} &  \mathcal{C}' \ar[d]_{s'} \ar[r]^{ev'} & S  \\
T \ar[r]^{h} & T' & }
\end{equation*}
that satisfies the following properties:
\begin{enumerate}
\item $\mathcal{C}'$ and $T'$ are normal.
\item $k(T')$ is algebraically closed in $k(\mathcal{C}')$.
\item $s'$ is a proper flat morphism such that the reduced subscheme underlying the fiber over a general closed point is a (possibly singular) irreducible rational curve.
\item $g$ and $h$ are finite morphisms.
\item $ev = ev' \circ g$ and $ev'$ is a separable map.
\end{enumerate}
We claim that every fiber of $s'$ over a general closed point of $T'$ is smooth.  First, \cite[Lemma 7.2]{Badescu01} shows that condition (2) above implies that $k(\mathcal{C}')$ is a separable extension of $k(T')$.  
In particular this implies that a general fiber $C'$ of $s'$ is reduced.  Next, by \cite[Proposition 5.2]{IIL20} the sum of the $\delta$-invariants at the closed points of a general fiber $C'$ is the same as the arithmetic genus.  Thus by \cite[Theorem 5.7]{IIL20} it suffices to show that the arithmetic genus of $C'$ is strictly less than $(p-1)/2$.  Since the map $ev$ takes a general fiber of $s$ birationally onto its image, the same is true of $ev'$.  This implies that the arithmetic genus of a fiber of $s'$ is at most the arithmetic genus of its image in $S$.  By the Hodge Index Theorem, a curve of anticanonical degree $\leq 3$ on a weak del Pezzo surface satisfies $dC^{2} \leq (-K_{S} \cdot C)^{2}$ and thus has arithmetic genus $< \frac{11-1}{2} = 5$.  Since by assumption $p \geq 11$ we conclude that the general fiber of $s'$ is smooth.

Since $ev'$ is a dominant separable morphism and the general fiber of $s'$ is smooth, we deduce that the restriction of the tangent bundle of $S$ to a general fiber of $s'$ is globally generated.  Since $g$ takes a general fiber of $s$ birationally onto the corresponding fiber of $s'$, the same property holds for the general fiber of $s$.
\end{proof}

As a corollary, we have the following statement:
\begin{coro} \label{coro:separabilitylargechar}
Let $k$ be an algebraically closed field of characteristic $p$. Let $S$ be a del Pezzo surface of degree $d$ over $k$. Assume that $p \geq \delta(d)$. When $d = 2$ and $p = 3$, we further assume that $S$ is not isomorphic to the surface listed in Theorem~\ref{theo:separablefamlowdegree}.(2). Then any rational curve of anticanonical degree $\leq 3$ containing a general point is a free rational curve.
\end{coro}

Let us add the following lemma for a later application:

\begin{lemm}
\label{lemm: line_conic}
Let $S$ be a del Pezzo surface of degree $1$ over an algebraically closed field $k$ of characteristic $p$.
%We assume that all components of $\overline{M}_{0,0}(S)$ parametrizing $-K_S$-lines and conics have the expected dimension and all dominant components of $\overline{M}_{0,0}(S)$ parametrizing $-K_S$-conics are separable. 
Assume that $p \geq 11$. Then a general member of a dominant family of $-K_S$-conics meets with any $-K_S$-line transversally and meets with any $-K_S$-conic transversally.
\end{lemm}

\begin{proof}
First of all, note that by Theorem~\ref{theo:separabilitylargechar} all components of $\overline{M}_{0,0}(S)$ parametrizing $-K_S$-lines and conics have the expected dimension and all dominant components of $\overline{M}_{0,0}(S)$ parametrizing $-K_S$-conics are separable. 

Let $C_1$ be a $-K_S$-conic and $C_2$ be a $-K_S$-line. Then note that $C_1^2 = 0$, $2$, or $4$. We also have $C_2^2 = -1$ or $1$.  By the Hodge Index Theorem the determinant of the intersection matrix of $-K_S, C_1, C_2$ is non-negative (regardless of the rank of this matrix).  Combining this fact with the above consideration, we conclude that $C_1 \cdot C_2 < 11$.

Let $p : \mathcal C \to N$ be a component of $\overline{M}_{0,0}(S)$ parameterizing $C_1$ with the evaluation map $f : \mathcal C \to S$.  A general curve parametrized by $N$ is free and so after shrinking $N$ we may assume that the separable morphism $f$ is unramified. Indeed, after shrinking $N$ so that $N$ only parametrizes free curves, the evaluation map $f$ is \'etale by \cite[II.3.5.4 Corollary]{Kollar} and a flat descent argument. Thus we conclude that $f^{-1}(C_2)$ is reduced.  Since $C_1 \cdot C_2  < 11$, our assumption on the characteristic implies that every dominant component of $f^{-1}(C_2)$ maps separably to $N$, so the intersection of $f^{-1}(C_2)$ and a general fiber of $p$ is reduced. This implies that a general $C_1$ meets with $C_2$ transversally.

Next let $C_1, C_2$ be two $-K_S$-conics. Then one can prove that $C_1.C_2 < 11$.  Repeating the argument above, we obtain transversality for a general conic meeting a conic.
\end{proof}

The statement of Lemma \ref{lemm: line_conic} fails in characteristic $2$:

\begin{exam}
It is possible on a del Pezzo surface $S$ that there is a fixed $-K_{S}$-line which is tangent to every $-K_S$-conic in a given $1$-dimensional family. For example, let $k$ be an algebraically closed field of characteristic $2$.  Consider the curve $C \subset \mathbb P^1_x \times \mathbb P^1_y$ defined by 
\[
x_0^2y_1 = x_1^2y_0.
\]
This is an integral smooth rational curve. Consider the morphism $\pi : \mathbb P^1 \times \mathbb P^1 \to \mathbb P^1$ mapping $(x, y)\mapsto y$. Then every fiber of $\pi|_C : C \to \mathbb P^1$ is non-reduced.
We blow up $5$ general points on $C$ and obtain a smooth cubic surface $\beta : S \to \mathbb P^1 \times \mathbb P^1$. Then the strict transform of $C$ is a $(-1)$-curve and every irreducible fiber of $\pi\circ\beta$ is a $-K_{S}$-conic which is tangent to $C$.
\end{exam}

%Finally one can prove the following lemma using an argument similar to the proof of Lemma~\ref{lemm: line_conic}:

%\begin{lemm}
%\label{lemm: conic_conic}
%Let $S$ be a del Pezzo surface of degree $1$ over an algebraically closed field $k$ of characteristic $p$.
%Assume that $p \geq 11$. Then a general member of a dominant family of $-K_S$-conics meets with any $-K_S$-conic transversally.
%\end{lemm}

%\begin{proof}
%Let $C_1, C_2$ be two $-K_S$-conics. Then one can prove that $C_1.C_2 < 11$.  Arguing as in Lemma \ref{lemm: line_conic}, this proves the claim.
%\end{proof}

\section{Inductive arguments using Bend and Break} \label{sect:inductionstep}

Let $k$ be an algebraically closed field (of arbitrary characteristic).  We would like to classify all dominant families of rational curves on a weak del Pezzo surface $S$ which either have larger than the expected dimension, or (more generally) fail to be separable.  Using Bend-and-Break we will reduce the classification problem to smaller degrees, eventually working downward to the base cases discussed in Section \ref{sect:expdimlowdegree} and Section \ref{sect:lowdegree}.

The following lemma is the key tool.

\begin{lemm} \label{lemm:weakmbbfordpsurfaces}
Let $S$ be a weak del Pezzo surface over $k$.  Fix a positive integer $d \geq 4$.  Assume that every irreducible component of $\overline{M}_{0,0}(S)$ that generically parametrizes a dominant family of birational maps onto irreducible curves with anticanonical degree $< d$ has the expected dimension.

Suppose that $M \subset \overline{M}_{0,0}(S)$ is an irreducible component that generically parametrizes a dominant family of birational maps onto irreducible curves of anticanonical degree $d$.  Fix $\dim(M)-1$ general points of $S$.  Then $M$ parametrizes a stable map $f: Z \to S$ where $Z$ has two different irreducible components $Z_{1},Z_{2} \subset Z$ such that $f(Z_{1}) \cup f(Z_{2})$ contains all $\dim(M)-1$ distinguished points and $f|_{Z_{1}}$, $f|_{Z_{2}}$ are general members of dominant families of birational stable maps in lower anticanonical degree.

If furthermore $S$ is a del Pezzo surface, then we can ensure that $Z_{1},Z_{2}$ are the only components of $Z$.
\end{lemm}

The proof is modeled on \cite[Lemma 1.14]{Testa09}.

\begin{proof}
Set $r = \dim(M)$.  If we fix $r-1$ general points of $S$ then by Bend-and-Break $M$ parametrizes a stable map $f$ with reducible domain whose image contains these $r-1$ points.  Furthermore, by \cite[Lemma 4.1]{LT19} (which works in arbitrary characteristic) we may find such a stable map $f$ such that there are at least two different irreducible components of the domain of $f$ such the image of each component contains one of our fixed points, and moreover the two points contained by the two components are different.  In particular, due to the generality of the points there must be at least two irreducible components of the domain of $f$ whose deformations dominate $S$.

Let $Z_{1},\ldots,Z_{s}$ be the irreducible components of the domain of $f$ whose images deform in a dominant family and let $C_{1},\ldots,C_{s}$ be their images in $S$.  The previous paragraph shows that $s \geq 2$.  By assumption every family of birational stable maps with irreducible domains of lower degree has the expected dimension.  
In particular, if we define $d_{i} := -K_{S} \cdot C_{i}$ then $C_{i}$ can contain at most $d_{i}-1$ general points of $S$.  On the other hand, we know that all $r-1$ general points must be contained in the image of $f$.  Since $r \geq d - 1 \geq (\sum d_{i}) - 1$, we see that there can be at most two such components $C_{i}$.  We conclude that $s=2$.  Furthermore, since each $C_{i}$ is going through the maximal possible number of general points in $S$, by choosing our points general we may ensure that $C_{1}$ and $C_{2}$ are general in their respective families. 
Since $d = d_{1} + d_{2}$ by the argument above, we must have that $f|_{Z_{i}}$ is birational for $i=1,2$.

Suppose furthermore that $S$ is a del Pezzo surface.  Since the argument above shows that $d = d_{1} + d_{2}$ we see that there can be no other curves in the image of $f$.
\end{proof}

We can now address the existence of families of rational curves with larger-than-expected dimension.

\begin{prop} \label{prop:expecteddim}
Let $S$ be a weak del Pezzo surface over $k$.  Assume that every dominant component of $\overline{M}_{0,0}(S)$ generically parametrizing birational maps to rational curves of anticanonical degree $\leq 2$ has the expected dimension.  Let $M \subset \overline{M}_{0,0}(S)$ be any component that generically parametrizes a dominant family of birational maps onto irreducible curves $C$.  Then $M$ has the expected dimension.
\end{prop}

\begin{proof}
We prove the statement by induction on the anticanonical degree.  The base case when $-K_{S} \cdot C \leq 2$ is true by assumption.

Suppose that $-K_{S} \cdot C = 3$.  If the deformations of $C$ had larger than the expected dimension, then by applying Bend-and-Break as in the proof of Lemma \ref{lemm:weakmbbfordpsurfaces} we see that $S$ must also carry a dominant family of rational curves of degree $\leq 2$ which has higher than the expected dimension.  This gives a contradiction.

Suppose that $-K_{S} \cdot C \geq 4$.  Set $r = \dim(M)$.  By Bend-and-Break we find a stable map parametrized by $M$ with reducible domain through $r-1$ general points of $S$.  By Lemma \ref{lemm:weakmbbfordpsurfaces} there are two curves $C_{1},C_{2}$ in the image of $f$ which deform in a dominant family and contain all $r-1$ general points.  Letting $d_{1},d_{2}$ denote the anticanonical degrees of the curves, we have
\begin{equation*}
r-1 \leq (d_{1}-1) + (d_{2}-1) \leq -K_{S} \cdot C - 2 \leq r-1
\end{equation*}
and thus $r = -K_{S} \cdot C - 1$.
\end{proof}

\begin{proof}[Proof of Theorem \ref{theo:maintheorem1}:]
Combine  Corollary \ref{coro:lowdegreeexpectdimanticanonical} and Proposition \ref{prop:expecteddim}.
\end{proof}

Next we consider whether or not families of high degree rational curves are separable.

\begin{prop}
\label{prop: free}
Let $S$ be a del Pezzo surface over $k$.  Assume that every dominant component of $\overline{M}_{0,0}(S)$ generically parametrizing birational maps to rational curves of anticanonical degree $\leq 3$ is separable. 
Let $M \subset \overline{M}_{0,0}(S)$ be any component that generically parametrizes a dominant family of birational maps onto irreducible curves $C$. Then $M$ generically parametrizes a free curve.
\end{prop}

\begin{proof}
Let $C$ be a general member of $M$ and let $-K_S \cdot C = d$. We prove our statement by induction on $d$.  By assumption the desired statement holds when $d \leq 3$.

When $d \geq 4$, we apply Bend-and-Break and Lemma \ref{lemm:weakmbbfordpsurfaces} to find a stable map $f:Z \to S$ parametrized by $M$ whose domain has exactly two irreducible components.  Furthermore Lemma \ref{lemm:weakmbbfordpsurfaces} guarantees that the images $C_{1},C_{2}$ are general in their respective families, hence free.  

Let $B$ be a general curve in $\overline{M}_{0,0}(S)$ through $f$.  After perhaps replacing $B$ by a cover, we obtain a universal family $\mathcal{U}$ over $B$ equipped with a map $g: \mathcal{U} \to S$ such that the central fiber is $Z$ and $g|_{Z} = f$.  Since a general fiber of $\mathcal{U} \to B$ is isomorphic to $\mathbb{P}^{1}$, this map admits a section.  Thus there is a line bundle $\mathcal{L}$ on $\mathcal{U}$ which has degree $-1$ against one component of $Z$ and degree $0$ against the other.  Then we have $H^{1}(Z,g^{*}T_{S} \otimes \mathcal{L}|_{Z}) = 0$.  By upper semicontinuity of cohomology groups we deduce that the general map parametrized by $M$ is free. Indeed, let $h: C \to S$ be a general stable map parametrized by $B$. Then $H^1(C, h^*T_S\otimes O(-1)) = 0$. Thus if we write $h^*T_S = \mathcal O(a_1) \oplus \mathcal O(a_2)$, then we must have $a_i \geq 0$ proving that $h:C \to S$ is free.
\end{proof}

\begin{proof}[Proof of Theorem \ref{theo:maintheorem2}:]
Combine Theorem \ref{theo:separablefamlowdegree} and Proposition \ref{prop: free}.
\end{proof}

\section{Del Pezzo surfaces of degree 1 in characteristic 0}  \label{sect:dpchar0}

In this section we work over a fixed algebraically closed field $k$ of characteristic $0$.
Let $S$ be a del Pezzo surface over $k$ and $\beta$ a class in $N_1(S)_\bZ$. Denote by $\overline{M}^{bir}(S,\beta)$ the closure of the locus in the Kontsevich space $\overline{M}_{0,0}(S,\beta)$ parametrizing generically injective maps with irreducible domains.  Our goal in this section is to prove the irreducibility of $\overline{M}^{bir}(S,\beta)$.

\cite{Testathesis} proved this result when $S$ has degree $\geq 2$ or when $S$ is general of degree $1$:

\begin{theo}[\cite{Testathesis} Section 2.2 and Theorem 4.5]
Let $S$ be a del Pezzo surface of degree $d$ over an algebraically closed field of characteristic $0$.  Suppose that either $d \geq 2$ or $d=1$ and $S$ is general in moduli.  Then for every numerical class $\beta$ on $S$ the scheme $\overline{M}^{bir}(S,\beta)$ is either irreducible or empty.
\end{theo}

We focus on the last case of arbitrary del Pezzo surfaces of degree $1$.  We will use the following result:

\begin{prop}[\cite{Testathesis} Proposition 4.6]
Let $S$ be a del Pezzo surface over an algebraically closed field of characteristic $0$.  Suppose that $\beta$ is a nef numerical class on $S$.  If $\beta$ is not the multiple of a $-K_{S}$-conic, then $\overline{M}^{bir}(S,\beta)$ is non-empty.
\end{prop}

We will prove the irreducibility of $\overline{M}^{bir}(S,\beta)$ by deforming to a general del Pezzo surface of degree $1$.  The key construction is the following:

\begin{lemm}\label{curve-general}  
Suppose that $S$ is a del Pezzo surface over an algebraically closed field of characteristic $0$ and $\beta \in N_1(S)_\bZ$ is such that $e:=-K_S \cdot \beta \geq 3$ and $\overline{M}^{bir}(S,\beta)$ is non-empty.  Let  $q_1, \dots, q_{e-2}$ be general points in $S$ and let 
$B$ be the locus in  $\overline{M}^{bir}(S,\beta)$ parametrizing morphisms whose images pass through $q_1, \dots, q_{e-2}$. Then $B$ is of dimension 1 and lies in the smooth locus of $\overline{M}_{0,0}(S,\beta)$.
There are 
finitely many maps parametrized by $B$ with reducible domains. 
\end{lemm}

%((
\begin{proof}
Let $f: Z \to S$ be a stable map parametrized by $B$. If $Z$ is irreducible, then $f$ is free and so $(Z,f)$ a smooth point of the moduli space. Suppose $Z$ is reducible. 
Let $Z_1, \dots, Z_m$, $m \geq 2$ be the irreducible components of $Z$ not contracted by $f$ and let $e_i = \deg f|_{Z_i}$ and $f_i=f|_{Z_i}$. 
Suppose $f_1, \dots, f_k$ are free maps and $f_{k+1}, \dots, f_m$ are non-free. Then the image of $Z_i$, $1 \leq i \leq k$, can pass through at most $e_{i}-1$ general points. So $e-2 \leq e_1 + \dots +e_k -k$. On the other hand $\sum_{i=1}^m e_i =e$, so there are two possibilities: either 1) $m=k=2$ or 2) $k=1, m=2, e_1= e-1, e_2=1$.  And in either case there cannot be a contracted component. In the first case, the image of $Z_i$, $i=1,2$, has to pass through $e_i-1$ of the points, and so there are finitely many choices for each $f_i$. Since 
the images of $f_1$ and $f_2$ pass through general points, they are free and so $f$ is a smooth point of the moduli space. 

In the second case, the image of $Z_2$ is a $-K_S$-line, and the image of $Z_1$ passes through $q_1, \dots, q_{e-2}$, so $f_1$ is general in its moduli and $N_{f_1}=\mathcal O(e-3)$. Since there are finitely many lines on $S$, \cite[Proposition 2.8]{BLRT20} shows that the images of $Z_1$ and $Z_2$ meet transversally.
Thus $f$ is a local immersion in an open neighborhood of the node of $Z$. If the image of $Z_2$ is a $(-1)$-curve or a nodal curve, 
then $N_{f_1}=\mathcal O(e-3)$ and $N_{f_2}= \mathcal O(-1)$, so $N_f|_{Z_1} = \mathcal O(e-2)$, and $N_f|_{Z_2}=\mathcal O$. If the image of $Z_2$ has a cusp, then 
$N_{f_1} = \mathcal O(e-3)$ and $N_{f_2}= \mathcal O(-2) \oplus k(p)$ where $p$ is the point at which $f$ is ramified. Therefore, $N_f|_{Z_1}=\mathcal O(e-2)$, and 
$N_f|_{Z_2} = \mathcal O(-1) \oplus k(p)$.  In both cases $(Z,f)$ is a smooth point of the moduli space. 
\end{proof}

\begin{prop}\label{connected-curve}
Suppose $S$ is a del Pezzo surface over an algebraically closed field of characteristic $0$ of degree $9-d \geq 1$. Fix $\beta \in N_1(S)_\bZ$ such that $e:=-K_S \cdot \beta \geq 3$. Then for general points $q_1, \dots, q_{e-2}$ 
on $S$, the locus $B$  in  $\overline{M}^{bir}(S, \beta)$ parametrizing morphisms whose images pass through $q_1, \dots, q_{e-2}$ is either empty or a connected curve.
\end{prop}

\begin{proof}
Fix a blow-down map $\pi: S \to \mathbb{P}^2$.  Write $\beta = mH - k_1E_1-\dots-k_{d}E_d$ where $H$ is the pull-back of the hyperplane class via $\pi$, the $E_i$ are $\pi$-exceptional divisors, and 
$k_i \geq 0$. First suppose $S$ is general. By \cite{Testathesis} $\overline{M}^{bir}(S, \beta)$ is irreducible. Let $U$ be the open subset of $\overline{M}^{bir}(S, \beta)$  parametrizing 
generically injective morphisms from $\mathbb{P}^1$ to $S$. By composing with the blow-down $\pi: S \to \mathbb{P}^2$ we get an embedding from $U$ to the Hilbert scheme of curves of degree $m$ in $\mathbb{P}^2$. We let $\mathbb{P}^N$ denote the projective space of 
curves of degree $m$ in $\mathbb{P}^2$, so we get a morphism $U \to \mathbb{P}^N$
and thus a rational map $\alpha: \overline{M}^{bir}(S, \beta) \dashrightarrow \mathbb{P}^N$. Resolving the indeterminacy locus of $\alpha$, we get  morphisms $\widetilde{\alpha}: \widetilde M \to \mathbb{P}^N$ and $p: \widetilde{M} \to \overline{M}^{bir}(S, \beta)$ such that $\widetilde{\alpha}= \alpha \circ p$.  The image of $\alpha$ is $(e-1)$-dimensional and the 
images of $q_1, \dots, q_{e-2}$ in $\mathbb{P}^2$ give a linear subvariety $\Lambda$ of codimension $ \leq e-2$ in $\mathbb{P}^N$ parametrizing curves of degree $m$ passing through them. 
By \cite[Theorem 2.1]{FL81}, $\widetilde{\alpha}^{-1}(\Lambda)$ is connected.  
By the above lemma, a general point in every irreducible component of $B$ has an irreducible domain and is therefore in the domain of $\alpha$. 
Since $B= p(\widetilde{\alpha}^{-1}(\Lambda))$, we conclude that $B$ is connected as well. 

Now suppose that $S$ is an arbitrary del Pezzo surface of degree $9-d$.  Let $V_d$ be the open subvariety of $\Hilb^{d} (\mathbb{P}^2)$ parametrizing 
$d$ points in general position in $\mathbb{P}^2$ (in the usual sense for del Pezzo surfaces), $Z$ the universal Hilbert scheme over $V_d$ and $\mathcal S$ the blow-up of $\mathbb{P}^2 \times V_d$ with center  
$Z$. Then there is a point $u \in V_d$ such that $S= \mathcal S_u$. Denote by $q'_1, \dots, q'_{e-2}$ the images of $q_1, \dots, q_{e-2}$ in $\mathbb{P}^2$. For any map $f: \mathbb{P}^1 \to \mathcal S_u$ whose image passes through $q_1, \dots, q_{e-2}$ we have
$H^1(N_f(-e+2))=0$, so the deformations of $f$ yield a family of stable maps from $\mathbb{P}^1$ to the fibers of $\mathcal{S} \to V_{d}$ passing through the pre-images of $q'_1, \dots, q'_{e-2}$ which has the expected dimension.
Let $\mathcal M$ be the family of stable maps to fibers of $\mathcal S \to V_d$ passing through the preimages of $q'_1, \dots, q'_{e-1}$.  We claim that every irreducible component of $\mathcal{M}$ dominates $V_{d}$.  This follows from a dimension calculation: we know that each fiber of $\mathcal{M} \to V_{d}$ is at most $1$-dimensional.  A normal bundle calculation shows that each component of $\mathcal{M}$ has dimension at least $\dim(V_{d}) + 1$.  Together these observations prove the claim.

Consider the map from $\mathcal M$ to $V_{d}$.  Since the general fiber of this map is connected and every component of $\mathcal{M}$ dominates $V_{d}$, we see that $\mathcal{M}$ is also connected.  Since $V_d$ is smooth and the fiber over a general point of $V_d$ is connected, the Stein factorization of the proper map $\mathcal{M} \to V_{d}$ is trivial and thus the fiber over every closed point $u$ is connected. 
\end{proof}

\begin{theo}\label{c-irreducible}
If $S$ is a smooth del Pezzo surface of degree $1$ over an algebraically closed field of characteristic $0$, then for every $\beta \in N_1(S)_\bZ$ with $-K_S \cdot \beta \geq 3$, $\overline{M}^{bir}(S, \beta)$ is either irreducible or empty.
\end{theo} 

\begin{proof}
Suppose that $\overline{M}^{bir}(S, \beta)$ is non-empty.
Let $e = -K_S \cdot \beta$, and pick $e-2$ general points $q_1, \dots, q_{e-2}$ in $S$. Lemma \ref{curve-general} shows that in every component of $\overline{M}^{bir}(S, \beta)$ there is a $1$-parameter family of curves parametrizing curves through $q_1, \dots, q_{e-2}$.  The union $B$ of all such $1$-parameter families is connected by Proposition \ref{connected-curve}.  Suppose that $\overline{M}^{bir}(S, \beta)$ is reducible and let $M_{1},\ldots,M_{k}$ denote the irreducible components.  Since $B$ is connected and $M_{i} \cap B \neq \emptyset$ for every component $M_{i}$, we see that there must be a point $b \in B$ which is contained in two different irreducible components.  In particular, $b$ must be a singular point of $\overline{M}_{0,0}(S, \beta)$.  But this is not possible by Lemma \ref{curve-general}.
\end{proof}

If $\beta$ is a multiple of a $-K_{S}$-conic, then it is easy to see that $\overline{M}_{0,0}(\beta)$ will admit a component that generically parametrizes multiple covers of the corresponding conic fibration.  Altogether we have:

\begin{theo} \label{theo:dp1char0classification}
Let $S$ be a smooth del Pezzo surface of degree $1$ over an algebraically closed field of characteristic $0$.  Let $\beta$ be a nef class on $S$ satisfying $-K_S \cdot \beta \geq 3$.  Then:
\begin{enumerate}
\item If $\beta$ is not a multiple of a $-K_{S}$-conic, then there is a unique component $M$ of $\overline{M}_{0,0}(S,\beta)$ generically parametrizing stable maps with irreducible domains.  The general map parametrized by $M$ is a birational map onto a free curve.
\item If $\beta$ is a multiple of a smooth rational conic, then there is a unique component of $\overline{M}_{0,0}(S,\beta)$ generically parametrizing stable maps with irreducible domains.  The general map parametrized by $M$ is a finite cover of a smooth conic.
\item If there is a contraction of a $(-1)$-curve $\phi: S \to S'$ such that $\beta$ is the pullback of $-K_{S'}$, then there are exactly two components of $\overline{M}_{0,0}(S,\beta)$ parametrizing stable maps with irreducible domains.  One component generically parametrizes birational maps onto free curves, the other generically parametrizes multiple covers of conics.
\item If $\beta$ is a multiple of $-2K_{S}$ then there are at least two components of $\overline{M}_{0,0}(S,\beta)$ parametrizing stable maps with irreducible domains.  There is a unique component generically parametrizes birational maps onto free curves, and the others generically parametrize multiple covers of conics.
\end{enumerate}
\end{theo}

\begin{proof}
By \cite[Proposition 4.6]{Testathesis} there exists a free curve of class $\beta$, and thus a component $M \subset \overline{M}_{0,0}(S,\beta)$ generically parametrizing free curves.

Suppose that the general curve parametrized by $M$ is not birational onto its image.  If we let $m$ denote the anticanonical degree of the (reduced) image and $b$ the degree of the general map parametrized by $M$, then
\begin{equation*}
\dim(M) = (m-1) + (2b-2).
\end{equation*}
Since we also know that $M$ has at least the expected dimension $mb-1$, we deduce that $m=2$.  Thus if $M$ generically parametrizes non-birational maps, the images of these maps must be conics.  Conversely, since every nef class $\alpha$ satisfying $-K_{S} \cdot \alpha = 2$ is represented by a free conic, multiple covers of conics will always yield a component of $\overline{M}_{0,,0}(S,m\alpha)$.  Note that if $\alpha$ is the class of a smooth rational conic, then the moduli space of conics of class $\alpha$ is irreducible.  Similarly, if $\alpha$ is the pullback of the anticanonical divisor on a degree $2$ del Pezzo under a birational map, then the moduli space of conics of class $\alpha$ is the dual curve of the branch divisor for the induced map to $\mathbb{P}^{2}$ and thus must be irreducible. 

It only remains to analyze the case when $M$ generically parametrizes birational maps.  By Theorem \ref{c-irreducible} we know that $\overline{M}^{bir}(S,\beta)$ is either irreducible or empty.  We also know that $\beta$ is represented by a stable map with an irreducible domain by \cite[Proposition 4.6]{Testathesis}.  Thus we obtain the desired property if $\beta$ is not a multiple of a conic.  Since a smooth rational conic is a fiber of a morphism to $\mathbb{P}^{1}$, it is clear that a multiple of a smooth rational conic is not represented by any irreducible rational curves.  If $\beta$ is pulled back from a degree $2$ del Pezzo surface $S'$, then by gluing free curves representing $|-K_{S'}|$, smoothing, and taking a strict transform we find an irreducible rational curve of class $\beta$.  If $\beta$ is a multiple of $|-2K_{S}|$, then by gluing free curves in $|-2K_{S}|$ and smoothing we see that $\beta$ is represented by an irreducible rational curve. 
\end{proof}

\section{Irreducibility of moduli spaces in characteristic $p$} \label{sect:irrcharp}

Let $S$ be a del Pezzo surface defined over an algebraically closed field $k$ of characteristic $p$.
Denote by $\overline{M}^{bir}(S,\beta)$ the closure of the locus in the Kontsevich space $\overline{M}_{0,0}(S,\beta)$ 
parametrizing generically birational maps with irreducible domains.  As in the previous section, our goal is to show that $\overline{M}^{bir}(S,\beta)$ is irreducible under suitable hypotheses.  Our strategy is to deform to characteristic $0$.

\subsection{Existence of stable maps with irreducible domains}

We first need to show the existence of stable maps with irreducible domains which map birationally onto their image.  We will mimic the approach of \cite{Testathesis}.  The first step is:

\begin{lemm}[\cite{Testathesis} Corollary 2.5] \label{lemm:nefdecomposition}
Let $S$ be a del Pezzo surface of degree $d \leq 8$ over an algebraically closed field.  Let $D$ be a nef Cartier divisor on $X$.  Then there is a sequence of contractions of $(-1)$-curves
\begin{equation*}
S = Y_{d} \to Y_{d+1} \to \ldots \to Y_{8},
\end{equation*}
non-negative integers $n_{d},n_{d+1},\ldots,n_{7}$, and a nef divisor $D'$ on $Y_{8}$ such that
\begin{equation*}
D = n_{d}(-K_{Y_{d}}) + n_{d+1}\phi_{d+1}^{*}(-K_{Y_{d+1}}) + \ldots + n_{7} \phi_{7}^{*}(-K_{Y_{7}}) + \phi_{8}^{*}D'
\end{equation*}
where $\phi_{d+i}: Y_{d} \to Y_{d+i}$ is the composition of the birational maps in the above sequence.
\end{lemm}

\begin{proof}
Recall that the description of del Pezzo surfaces as blow-ups of $\mathbb{P}^{2}$ is exactly the same in characteristic $p$ and characteristic $0$.  Since the proof of \cite[Corollary 2.5]{Testathesis} only uses the combinatorics of these blow-ups, the proof works equally well in any characteristic.
\end{proof}

\begin{prop}[\cite{Testathesis} Proposition 4.6] \label{prop:existfreecurve}
Let $S$ be a del Pezzo surface over an algebraically closed field. When the degree of $S$ is $1$ we assume that the characteristic of the ground field is not equal to $2$. 
We further assume that every dominant component of $\overline{M}_{0,0}(S)$ that generically parametrizes birational maps to rational curves of anticanonical degree $\leq 3$ is separable.   
 Then every nef class $\alpha$ is represented by a stable map $f: \mathbb{P}^{1} \to S$ which is a free curve.
\end{prop}

\begin{proof}
First we must address the nef classes $\alpha$ satisfying $-K_{S} \cdot \alpha = 1$.  By the Hodge Index Theorem this can only occur when $S$ has degree $1$ and $\alpha = -K_{S}$.  Note that $\mathbb{P}^{2}$ contains a cubic rational curve through any 8 points.  Since $-K_{S}$ is primitive, by taking a strict transform we see that $\alpha$ is represented by an irreducible rational curve.

The more interesting case is when $-K_{S} \cdot \alpha \geq 4$.  As in Lemma \ref{lemm:nefdecomposition} we can write
\begin{equation*}
D = n_{d}(-K_{Y_{d}}) + n_{d+1}\phi_{d+1}^{*}(-K_{Y_{d+1}}) + \ldots + n_{7} \phi_{7}^{*}(-K_{Y_{7}}) + \phi_{8}^{*}D'
\end{equation*}
where each $Y_{i}$ is a del Pezzo surface of degree $i$.

We claim that if $i \geq 4$ then $|-K_{Y_{i}}|$ is represented by a free rational curve. The existence of an irreducible rational curve in the anticanonical linear system follows from the fact that these are the strict transforms of plane cubics passing through the points we blow up.  A general member is free by Proposition~\ref{prop: free}. 

Note that by assumption $|-K_{Y_{2}}|$ and $|-K_{Y_{3}}|$ are also represented by free rational curves (if $d \leq 3$).

We now construct a chain of rational curves representing $D$.  If $d \geq 2$, we construct the chain by taking $n_{d}$ general free curves in $|-K_{Y_{d}}|$, then connecting it to a chain of $n_{d+1}$ general free curves in $\phi_{d+1}^{*}|-K_{Y_{d+1}}|$, and so on until we reach $Y_{7}$.  Since a del Pezzo surface of degree $8$ is either $\mathbb{P}^{1} \times \mathbb{P}^{1}$ or the blow-up of $\mathbb{P}^{2}$ at a point, it is also clear that $D'$ is represented by a free rational curve.  Altogether, if $d \geq 2$ then $D$ is represented by a stable map which maps birationally onto a chain of free rational curves.  By smoothing we obtain a stable map with irreducible domain mapping to a free curve.

When $d=1$ and $n_1 = 0$ the argument is similar.  If $n_{1} \geq 2$, then we can write $n_{1} = 2m_{1} + 3m_{2}$ for some non-negative integers $m_{1},m_{2}$.  By assumption $|-2K_{S}|$ and $|-3K_{S}|$ are both represented by free rational curves, and we conclude by a similar argument as before.
Finally, if $n_{1} = 1$, $C \in |-K_{Y_1}|$ has the normal sheaf $\mathcal O(-1)$ or $\mathcal O(-2)\oplus k(p)$.  (Recall that we are assuming that the characteristic $\neq 2$ so that $\mathcal{O}(-3) \oplus k[t]/(t^2)$ is not possible.)
Note that each $-\phi_{i}^{*}K_{Y_{i}}$ for $i > 1$ and $\phi_{8}^{*}D'$ can be expressed as a positive sum of at least two $(-1)$-curves.  Thus we may represent the class $D$ as a comb whose handle is $C$ and whose teeth are a collection of $(-1)$-curves $E_j$.  Since $E_j \cdot C = 1$, the $(-1)$-curves meet $C$  transversally.  Since there are at least two $(-1)$-curves, Theorem~\ref{thm-normalBundleNodalCurve} shows that the resulting comb is a smooth point of the moduli space.  We can smooth it so that we obtain an irreducible free curve, proving the claim.
\end{proof}

\subsection{Deforming to characteristic $0$}

\begin{lemm}\label{curve-general_p}
Suppose that $S$ is a smooth del Pezzo surface of degree $d$ defined over an algebraically closed field $k$ of characteristic $p$.
Assume that $p \geq \delta(d)$. When $d = 2$ and $p = 3$,  we further assume that $S$ is not isomorphic to the surface listed in Theorem~\ref{theo:separablefamlowdegree}.(2).

Let $\beta \in N_1(S)_\bZ$ be a nef class such that $e:=-K_S \cdot \beta \geq 3$ and $\overline{M}^{bir}(S,\beta)$ is non-empty. Let  $q_1, \dots, q_{e-2}$ be general points in $S$ and let 
$B$ be the locus in $\overline{M}^{bir}(S,\beta)$ parametrizing stable maps whose images pass through $q_1, \dots, q_{e-2}$. 
Then $B$ is of dimension 1 and lies in the smooth locus of $\overline{M}_{0,0}(S,\beta)$. Furthermore only 
finitely many maps parametrized by $B$ have reducible domains.
\end{lemm}

\begin{proof}
Proposition \ref{prop:expecteddim} shows that every component of $\overline{M}^{bir}(S,\beta)$ has the expected dimension.  Since general points will impose general conditions on a family of curves, we see that $B$ has dimension $1$.

We prove the remaining statements by induction on $e$.
When $e = 3$, let $f: Z \to S$ be a stable map of degree $3$ passing through a general point.
If $Z$ is irreducible, then $(Z,f)$ is a smooth point of the moduli space by Corollary~\ref{coro:separabilitylargechar}.
Assume that $Z$ is reducible. Then $Z$ consists of a $-K_S$-conic $Z_1$ and a $-K_S$-line $Z_2$ with $f_i = f|_{Z_i}$.
Assume that $d \geq 2$. Then the image of $Z_2$ is a $(-1)$-curve, so it is smooth. It follows from Proposition~\ref{prop-vanH1NodalCurve} that $(Z,f)$ is a smooth point of the moduli space.
If $d = 1$, then it follows from Lemma~\ref{lemm: line_conic} that the images of $Z_1$ and $Z_2$ meet transversally. Thus we conclude that $f$ is a local immersion in an open neighborhood of the node of $Z$. If the image of $Z_2$ is a $(-1)$-curve or a nodal curve, 
then $N_{f_1}=\mathcal O$ and $N_{f_2}= \mathcal O(-1)$, so $N_f|_{Z_1} = \mathcal O(1)$, and $N_f|_{Z_2}=\mathcal O$. If the image of $Z_2$ has a cusp, then 
$N_{f_1} = \mathcal O$ and $N_{f_2}= \mathcal O(-2) \oplus k(p)$ where $p$ is the point at which $f$ is ramified. Therefore, $N_f|_{Z_1}=\mathcal O(1)$, and 
$N_f|_{Z_2} = \mathcal O(-1) \oplus k(p)$.  In both cases $(Z,f)$ is a smooth point of the moduli space.

We now prove the induction step.  Choose $e \geq 4$ and assume our assertion is true for stable maps of anticanonical degree $< e$.
Let $(Z, f)$ be a stable map parametrized by $B$. If $Z$ is irreducible, then we claim that $(Z,f)$ is a smooth point of the moduli space. Suppose otherwise, so that the singular locus of $\overline{M}_{0,0}(S)$ meets the curve $B$ at a point representing a map with irreducible domain.  As we vary the choice of $e-2$ general points $q_{1},\ldots,q_{e-2}$, the curves $B$ define a flat family of subvarieties of $\overline{M}^{bir}(S,\beta)$.   Since a flat family of subvarieties will intersect any other subvariety in the expected dimension, there must be a component $V$ of the singular locus of $\overline{M}_{0,0}(S)$ which has codimension $1$ in $\overline{M}_{0,0}(S)$ and generically parametrizes birational stable maps with irreducible domains.  
Pick general points $q_1', \cdots, q_{e-3}'$ and consider a $1$-dimensional locus of $V$ generically parametrizing irreducible curves passing through $q_1', \cdots, q_{e-3}'$. 
Arguing as in Lemma~\ref{lemm:weakmbbfordpsurfaces}, $f$ breaks into a stable map with reducible domain, and there are the following possible types of breaking curves $(Z', g)$, where the $Z'_{i}$ denote the irreducible components of $Z'$:
\begin{enumerate}
\item $Z' = Z'_1 \cup Z'_2$ with $-K_S \cdot Z'_1 = d_1>2$, $-K_S \cdot Z'_2 = d_2 > 1$ such that $Z'_1$ contains $d_1-2$ general points and $Z'_2$ contains $d_2 -1$ general points or;
\item $Z' = Z'_1 \cup Z'_2 \cup Z'_3$ with  $-K_S \cdot Z'_1 = d_1>1$, $-K_S \cdot Z'_2 = d_2 > 1$, and $-K_S \cdot Z'_3 = d_3 > 1$ such that each $Z'_i$ contains $d_i-1$ general points, or;
\item $Z' = Z'_1 \cup Z'_2 \cup Z'_3$ with  $-K_S \cdot Z'_1 = d_1>1$, $-K_S \cdot Z'_2 = d_2 > 1$, and $-K_S \cdot Z'_3 = 1$ such that $Z'_1$ contains $d_1-1$ general points and $Z'_2$ contains $d_2 -1$ general points. 
\end{enumerate}
In the first case, the induction hypothesis shows that $Z'_1$ and $Z'_2$ are smooth points of the moduli space. This implies that $h^1(Z'_i, g^*T_S|_{Z'_i}) = 0$. Furthermore, $Z'_2$ is general in its moduli so it must be free. Thus we conclude that $h^1(Z', g^*T_S) = 0$. Then $(Z', g)$ is a smooth point of $\overline{M}_{0,0}(S)$, a contradiction.
In the second case, the $Z_i$'s are general in moduli so they are free. Thus $(Z', g)$ is a smooth point of moduli space, a contradiction.
In the third case, $Z'_1$ and $Z'_2$ must be free. 
When $d \geq 2$, $Z'_3$ is a smooth curve. Hence it follows from Proposition~\ref{prop-vanH1NodalCurve} that $(Z', g)$ is a smooth point of the moduli space.
Assume that $d = 1$. We claim that the images of $Z'_1, Z'_2, Z'_3$ meet transversally with each other. When the degree of $Z'_1$ is greater than $2$, then $Z'_1$ is very free by the induction hypothesis. Thus transversality of $Z'_1$ with $Z'_2$ and $Z'_3$ is clear. Similarly for $Z'_2$, so without loss of generality we may assume that $Z'_1$ and $Z'_2$ are $-K_S$-conics.  
Then transversality follows from Lemma~\ref{lemm: line_conic}. Arguing as above, we conclude that $(Z', g)$ is a smooth point of the moduli space, a contradiction. Altogether, for a general choice of $q_1, \cdots, q_{e-2}$ points of the form $(Z, f)$ in $B$ with $Z$ irreducible are smooth points of the moduli space.

%((
Now suppose $Z$ is reducible. 
Let $Z_1, \dots, Z_m$, $m \geq 2$ be the non-contracted  irreducible components of $Z$, and let $e_i = \deg f|_{Z_i}$ and $f_i=f|_{Z_i}.$ 
Suppose $f_1, \dots, f_k$ are the maps containing at least one of the general points and $f_{k+1}, \dots, f_m$ are the maps containing none of the general points. Then for $1 \leq i \leq k$ the image of $Z_i$ can pass through at most $e_{i}-1$ general points. So $e-2 \leq e_1 + \dots +e_k -k$. On the other hand $\sum_{i=1}^m e_i =e$, so there are two possibilities: either 1) $m=k=2$ or 2) $k=1, m=2, e_1= e-1, e_2=1$.  And in either case there cannot be a contracted component. In the first case, the image of $Z_i$, $i=1,2$, has to pass through $e_i-1$ of the points, and so there are finitely many choices for each $f_i$. Since 
the images of $f_1$ and $f_2$ pass through the maximum number of general points, they are free and so $(Z,f)$ is a smooth point of the moduli space.

In the second case, the image of $Z_2$ is a $-K_S$-line, and the image of $Z_1$ passes through $q_1, \dots, q_{e-2}$, so $f_1$ is general in its moduli and $N_{f_1}=\mathcal O(e-3)$. 
If $d \geq 2$, then $(Z, f)$ is a smooth point of the moduli space as above.
Suppose that $d = 1$. Since there are finitely many lines on $S$ the images of $Z_1$ and $Z_2$ meet transversally.
Indeed, assume to the contrary that $Z_1$ is tangent to $Z_2$. Since $Z_1$ is general in its moduli, this is only possible when $Z_1$ is a $-K_S$-conic. However this contradicts with our assumption on $\mathrm{char}(k)$ and Lemma~\ref{lemm: line_conic}.

Thus we conclude that $f$ is a local immersion in an open neighborhood of the node of $Z$. If the image of $Z_2$ is a $(-1)$-curve or a nodal curve, 
then $N_{f_1}=\mathcal O(e-3)$ and $N_{f_2}= \mathcal O(-1)$, so $N_f|_{Z_1} = \mathcal O(e-2)$, and $N_f|_{Z_2}=\mathcal O$. If the image of $Z_2$ has a cusp, then 
$N_{f_1} = \mathcal O(e-3)$ and $N_{f_2}= \mathcal O(-2) \oplus k(p)$ where $p$ is the point at which $f$ is ramified. Therefore, $N_f|_{Z_1}=\mathcal O(e-2)$, and 
$N_f|_{Z_2} = \mathcal O(-1) \oplus k(p)$.  In both cases $(Z,f)$ is a smooth point of the moduli space. 
\end{proof}

\begin{theo} \label{theo:charpbirorempty}
Let $S$ be a smooth del Pezzo surface of degree $d$ over an algebraically closed field $k$ of characteristic $p$, and let $\beta$ be a nef curve class of anti-canonical degree $e \geq 3$.  We assume that $p \geq \delta(d)$. When $d = 2$ and $p = 3$, we further assume that $S$ is not isomorphic to the surface listed in Theorem~\ref{theo:separablefamlowdegree}.(2).
Then $\overline{M}^{bir}(S,\beta)$ is irreducible or empty. 
\end{theo}

\begin{proof}
Suppose  $\overline{M}^{bir}(S,\beta)$ is non-empty.  We may assume that $S$ is defined over a subfield $k' \subset k$ such that $k'$ is finitely generated over the prime field $\mathbb F_p$ and $N^1(S)_\bZ = N^1(S\otimes k)_\bZ$.  After replacing $k'$ by a finite extension inside $k$ we may assume that $S$ is $k'$-rational and in particular that $S(k')$ is Zariski dense in $S$.  After taking another finite extension of $k'$ inside $k$ if necessary, there is a normal complete local ring $R$ which is of finite type over $\mathbb Z_p$ with residue field 
$k'$ and generic point $\eta$ and a smooth surface $\mathcal S$ over $\Spec \, R$ such that $\mathcal  S \otimes_R k' = S$.  Let $F_1 \subset S^{e-2}$ be a proper closed subset containing all sets of $e-2$ closed points which fail to be general in the sense of Lemma~\ref{curve-general_p} when applied to $S_{\overline{k'}}$.
Next let $F_2 \subset (\mathcal S\otimes_R K(\eta))^{e-2}$ be a proper closed subset containing all sets of $e-2$ closed points which fail to be general in the sense of Lemma~\ref{curve-general_p} when applied to $\mathcal{S} \otimes_{R} \overline{K(\eta)}$.
We take the Zariski closure $\mathcal F_2 \subset \mathcal S \times_R \cdots \times_R \mathcal S$ of $F_2$.
We define $\mathcal U$ as the Zariski open subset of $\mathcal S \times_R \cdots \times_R \mathcal S$ which is the complement of $F_1 \cup \mathcal F_2$.

Choose points $q_1, \dots, q_{e-2}$ defined over $k'$ in $S$ whose product lies in $\mathcal{U}$.  Since $S$ is smooth, we may apply Hensel's lemma to find sections $\tilde q_1, \dots, \tilde q_{e-2}$ of $\mathcal S \to \Spec \, R$ such that $\tilde{q}_i \otimes_R k' = q_i$.  By construction the product of the $e-2$ points $\tilde{q}_i\otimes_R K(\eta)$'s in $\mathcal S\otimes_R K(\eta)$ is contained in $\mathcal U$.

We will write $\beta_{\mathcal{S}} \in N_1(\mathcal S/\Spec \, R)_{\bZ}$ for the image of $\beta$ under the pushforward $N_1(S)_\bZ \to N_1(\mathcal S/\Spec \, R)_\bZ$.   Let $\widetilde M$ be the locus in $\overline{M}_{0,0}(\mathcal S/\Spec \, R,\beta_{S})$ parametrizing stable maps $f$ whose images meet with the images of  $\tilde q_1, \dots, \tilde q_{e-2}$.  Then by \cite[II.1.7 Theorem]{Kollar} the dimension of $\widetilde M$ is greater than or equal to
\begin{equation*}
1 + \dim \, R.
\end{equation*}
Indeed, \cite[II.1.7 Theorem]{Kollar} implies that a component $N$ of $\overline{M}_{0,0}(\mathcal S/\Spec \, R,\beta_{\mathcal{S}})$ which contains a component of $\widetilde{M}$ has dimension greater than or equal to $e -1 + \dim \, R$. We consider a component $N^{(e-1)} \subset \overline{M}_{0,e-1}(\mathcal S/\Spec \, R,\beta_{\mathcal{S}})$ above $N$ and the evaluation map $\mathrm{ev}_{e-1}: N^{(e-1)} \to \mathcal{S}^{e-1}$. Then $\widetilde{M}$ is the preimage of the product of the images of $\tilde{q}_1, \cdots, \tilde{q}_{e-2}$.  We conclude that $\widetilde{M}$ has dimension greater than or equal to $1 + \dim \, R$.
On the other hand every fiber of $\widetilde M \to \Spec \, R$ has at most dimension $1$ because of Lemma~\ref{curve-general_p}. Altogether, we have shown that every component of $\widetilde{M}$ dominantly maps to $\Spec \, R$.

By Proposition \ref{connected-curve}, the geometric generic fiber of $\widetilde M \to \Spec \, R$, i.e., $\widetilde M \otimes_R \overline{K(\eta)}$, is connected.  Since $\Spec \, R$ is normal, the Stein factorization of  $\widetilde M \to \Spec \, R$ is trivial so that all the geometric fibers are connected.  In particular the geometric fiber of $\widetilde{M}$ over the closed point of $R$ is connected.  Lemma~\ref{curve-general_p} shows that every point of this fiber is contained in the smooth locus of $\overline{M}_{0,0}(\mathcal{S}_{\overline{k'}})$.  Then the same argument as in Theorem \ref{c-irreducible} shows that $\overline{M}^{bir}(\mathcal{S}_{\overline{k'}},\beta)$ is irreducible.
Since $\mathcal{S}_{\overline{k'}}$ is constructed from $S$ by a base change of the ground field, our assertion follows.
\end{proof}

\begin{theo}
Let $S$ be a smooth del Pezzo surface of degree $d$ over an algebraically closed field of characteristic $p$.  Assume that $p \geq \delta(d)$. When $d = 2$ and $p = 3$,  we further assume that $S$ is not isomorphic to the surface listed in Theorem~\ref{theo:separablefamlowdegree}.(2).

Let $\beta$ be a nef class on $S$ satisfying $-K_S \cdot \beta \geq 3$.  Then:
\begin{enumerate}
\item If $\beta$ is not a multiple of a $-K_{S}$-conic, then there is a unique component $M$ of $\overline{M}_{0,0}(S,\beta)$ generically parametrizing stable maps with irreducible domains.  The general map parametrized by $M$ is a birational map onto a free curve.
\item If $\beta$ is a multiple of a smooth rational conic, then there is a unique component $M$ of $\overline{M}_{0,0}(S,\beta)$ generically parametrizing stable maps with irreducible domains.  The general map parametrized by $M$ is a finite cover of a smooth conic.
\item If $d=2$ and $\beta$ is a multiple of $-K_{S}$, or if $d = 1$ and there is a contraction of a $(-1)$-curve $\phi: S \to S'$ such that $\beta$ is a multiple of the pullback of $-K_{S'}$, then there are exactly two components of $\overline{M}_{0,0}(S,\beta)$ parametrizing stable maps with irreducible domains.  One component generically parametrizes birational maps onto free curves, the other generically parametrizes multiple covers of $-K_S$-conics.
\item If $d=1$ and $\beta$ is a multiple of $-2K_{S}$ then there are at least two components of $\overline{M}_{0,0}(S,\beta)$ parametrizing stable maps with irreducible domains.  There is a unique component generically parametrizes birational maps onto free curves, and the others generically parametrize multiple covers of $-K_S$-conics.
\end{enumerate}
\end{theo}

The proof is essentially the same as the proof of Theorem \ref{theo:dp1char0classification}.

\begin{proof}
By Proposition \ref{prop:existfreecurve} we know that $\overline{M}_{0,0}(S,\beta)$ is represented by a stable map with irreducible domain.  Let $M$ be a component generically parametrizing stable maps with irreducible domains.  When the general map parametrized by $M$ is not birational, we argue just as in the proof of Theorem \ref{theo:dp1char0classification}.  In particular this proves that such maps can only exist when $\beta$ is the multiple of a $-K_{S}$-conic.

It only remains to classify the irreducible components of $\overline{M}^{bir}(S,\beta)$.  By Theorem \ref{theo:charpbirorempty} $\overline{M}^{bir}(S,\beta)$ is either irreducible or empty.  If $\beta$ is not a multiple of the class of a conic, then Proposition \ref{prop:existfreecurve} shows that $\beta$ is represented by a stable map with irreducible domain and the previous paragraph shows that this map must be birational.  Thus $\overline{M}^{bir}(S,\beta)$ is non-empty, hence irreducible.  The case when $\beta$ is a multiple of a smooth rational conic is the same as in the proof of Theorem \ref{theo:dp1char0classification}.  If $d = 2$ and $\beta$ is a multiple of $-K_{S}$, recall that $|-K_{S}|$ is represented by a free curve by Theorem \ref{theo:separabilitylargechar}.  By gluing and smoothing a chain of such curves we find an irreducible rational curve of class $\beta$.  If $d=1$ and $\beta$ is a pullback under a map $\phi$ then we can find an irreducible rational curve of class $\beta$ by appealing to the degree $2$ case.  If $d = 1$ and $\beta$ is a multiple of $-2K_{S}$, then $|-2K_{S}|$ is represented by a free rational curve by Theorem \ref{theo:separabilitylargechar}.
\end{proof}

\section{The Fujita invariant for surfaces in characteristic $p$}
\label{sec: fujita}

In this section we study the Fujita invariant (which we will also call the $a$-invariant) for surfaces in characteristic $p$.  Our goal is to prove a classification theorem and to control the behavior of the Fujita invariant under finite covers.  Throughout we work over an algebraically closed field $k$ of characteristic $p$.

\begin{defi}\cite[Definition 2.2]{HTT15}
\label{defi: Fujita invariant}
Let $X$ be a smooth projective variety and let $L$ be a big and nef $\mathbb{Q}$-divisor on $X$. 
The {\it Fujita invariant} (which we will also call the $a$-invariant) is
$$
a(X, L) := \min \{ t\in \bR \mid t[L] + [K_X] \in \Eff^{1}(X) \}.
$$
If $L$ is nef but not big, we set $a(X,L) = \infty$.
\end{defi}

By \cite[Proposition 2.7]{HTT15}, $a(X, L)$ does not change when pulling back $L$ by a birational map between smooth varieties.  Thus, when $X$ is a singular projective variety which admits a resolution of singularities, we define the Fujita invariant for $X$ by pulling back to a smooth birational model $\phi : \widetilde{X} \ra X$:
$$
a(X, L):= a(\widetilde{X}, \phi^*L).
$$
This definition does not depend on the choice of $\phi$.

\begin{rema}
Suppose that $X$ is a smooth projective variety and $L$ is a big and nef divisor on $X$.  Then $a(X,L) > 0$ if and only if $X$ admits a dominant family of rational curves satisfying $K_{X} \cdot C < 0$.  This follows from the following theorem:
\end{rema}

\begin{theo}[\cite{MM86} Theorem 1, \cite{BDPP13} 0.3 Corollary, \cite{Das} Theorem 1.6]
\label{theo:Das}
Let $X$ be a smooth projective variety over an algebraically closed field.  Then $K_{X}$ is not pseudo-effective if and only if $X$ admits a dominant family of rational curves satisfying $K_{X} \cdot C < 0$.
\end{theo}

The rationality of the Fujita invariant is proved in characteristic $0$ for threefolds by Batyrev in \cite{Bat92} and for higher dimensional varieties in \cite{BCHM}.  For varieties of low dimension in characteristic $p$, it follows from the work of \cite{Das}.

\begin{theo}[\cite{Das}]
\label{theo:rationality}
Let $X$ be a smooth projective variety of dimension $\leq 3$ and $L$ be a big and nef $\mathbb Q$-divisor on $X$. We also assume that the characteristic $p$ of the ground field $k$ is $> 5$ when the dimension of $X$ is $3$.  Then $a(X, L)$ is rational.
\end{theo}

To derive this statement from \cite{Das}, we will need a well-known lemma:

\begin{lemm}
\label{lemm:terminalpair}
Let $X$ be a smooth projective variety of dimension $\leq 3$ and $L$ be a big and nef $\mathbb Q$-divisor on $X$. Let $a$ be any positive real number. Then  there exists an effective $\mathbb Q$-divisor $0 \leq L' \sim_{\mathbb Q} L$ such that $(X, aL')$ is a terminal pair.
\end{lemm}

\begin{proof}
It follows from \cite[Proposition 2.61]{KM98} that there exists an effective divisor $E$ such that for any rational number $0 < \epsilon \ll 1$, $A_\epsilon = L-\epsilon E$ is ample.
Let $\beta : \widetilde{X} \to X$ be a log resolution for $(X, E)$ whose existence is guaranteed by \cite{CP09} and \cite{Cut09} in dimension $3$.  Let $F$ be an effective exceptional divisor such that $-F$ is $\beta$-ample. 
Then we have $$\beta^*L = \beta^*A_\epsilon + \epsilon\beta^*E = \beta^*A_\epsilon - \epsilon F + \epsilon(\beta^*E  + F).$$
For $0 < \epsilon \ll 1$, $(\widetilde{X}, a\epsilon(\beta^*E  + F))$ is a terminal pair and $\beta^*A_\epsilon - \epsilon F$ is ample. Thus one can find a general ample $\mathbb Q$-divisor $A'_\epsilon \sim_{\mathbb Q} \beta^*A_\epsilon - \epsilon F$ such that the support of $A_\epsilon' + \epsilon(\beta^*E  + F)$ is a snc divisor and every coefficient of $aA'_\epsilon$ is strictly less than $1$. Let $L' = \beta_*A_\epsilon' + \epsilon E$. By the negativity lemma, we have $\beta^*L' = A_\epsilon' + \epsilon(\beta^*E  + F)$.  Thus when $\epsilon$ is sufficiently small, $(X, aL')$ is a terminal pair by \cite[Corollary 2.32]{KM98}.
\end{proof}

\begin{proof}[Proof of Theorem~\ref{theo:rationality}]
We only prove the case of dimension $3$.
After rescaling of $L$ we may assume that $a(X, L) > 1$.
We pick $L' = A'+ \epsilon E$ as in the proof of Lemma~\ref{lemm:terminalpair} with $a = a(X, L)$.
Let $V$ be the subspace of the space of $\mathbb R$-divisors which is generated by $A', E$.
Using the arguments in Lemma~\ref{lemm:terminalpair} one can find an ample $\mathbb Q$-divisor $A$ such that $A$ does not share any component with $A'$ and $E$, $A\sim_{\mathbb Q} A'$, and $K_X + A + (a-1)A' + a\epsilon E$ is terminal.
Then it follows from \cite[Theorem 1.2]{Das} that the pseudo-effective polytope $\mathcal E_{A}(V)$ in $V$ is a rational polytope.
Since $A + (a-1)A' + a\epsilon E + K_X$ is on the boundary of this polytope, we conclude that $a$ is rational.
\end{proof}

The following notion plays a central role in the study of Fujita invariants:
\begin{defi}
Let $X$ be a smooth projective variety of dimension $\leq 3$ such that $K_X$ is not pseudo-effective.
Let $L$ be a big and nef $\mathbb Q$-Cartier divisor on $X$. We say $(X, L)$ is adjoint rigid if $a(X, L)L + K_X$ has Iitaka dimension $0$.

When $X$ is singular and admits a resolution of singularities $\beta: \widetilde{X} \to X$, we say $(X, L)$ is adjoint rigid if $(\widetilde{X}, \beta^*L)$ is adjoint rigid.
This definition does not depend on the choice of $\beta$.  
\end{defi}

\begin{lemm} \label{lemm:ainvseparablecover}
Let $f: Y \to X$ be a dominant separable generically finite morphism of smooth varieties and let $L$ be a big and nef $\mathbb{Q}$-divisor on $X$.  Then $a(Y,f^{*}L) \leq a(X,L)$.
\end{lemm}

\begin{proof}
By the Riemann-Hurwitz formula there is an effective ramification divisor $R$ such that $K_{Y} = f^{*}K_{X} + R$.  Thus
\begin{equation*}
K_{Y} + a(X,L)f^{*}L = f^{*}(K_{X} + a(X,L)L) + R
\end{equation*}
is pseudo-effective, proving the desired inequality.
\end{proof}

Note that the result of Lemma \ref{lemm:ainvseparablecover} may fail for inseparable maps.  A well-known example is given by a unirational parametrization of a K3 surface: a smooth rational surface has positive $a$-invariant with respect to any polarization but a K3 surface has $a$-invariant $0$.

\subsection{Surfaces with large $a$-invariant}

We next classify the pairs of a smooth projective surface $S$ and a divisor $L$ such that $a(S,L) > 1$.  Since the minimal model program works just as in characteristic $0$, there are essentially no differences in the characteristic $p$ situation.  For completeness we will include a quick proof of every assertion.

\begin{prop}[\cite{LTT14} Proposition 5.9] \label{prop:surfaceconstants}
Let $S$ be a smooth uniruled projective surface over $k$ and let $L$ be a big and nef $\mathbb{Q}$-divisor on $S$.
\begin{enumerate}
\item Suppose that $\kappa(K_{S} + a(S,L)L) = 1$.  Let $F$ be a general fiber of the canonical map for $(S,a(S,L)L)$.  Then
\begin{equation*}
a(S,L) = a(F,L) = \frac{2}{L \cdot F}.
\end{equation*}
\item Suppose that $\kappa(K_{S} + a(S,L)L) = 0$.  Then there is a birational morphism $\phi: S \to S'$ where $S'$ is a smooth weak del Pezzo surface such that $-K_{S'} \sim_{\mathbb{Q}} a(S,L)\phi_{*}L$.
\end{enumerate}
\end{prop}

\begin{proof}
We run the miminal model program for $(S,a(S,L)L)$ to obtain a birational morphism $\phi: S \to S'$.  Since $L$ is a big and nef divisor each birational step of the MMP is a contraction of a $(-1)$-curve and the end result $S'$ is smooth.  We know that $K_{S'} + a(S,L)\phi_{*}L$ is semiample but not big so that its Iitaka dimension must be $0$ or $1$.  When the Iitaka dimension is $0$, we obtain the desired statement.  When the Iitaka dimension is $1$ then by the classification of surfaces we know that the corresponding map must have general fiber isomorphic to $\mathbb{P}^{1}$. Indeed, let $\pi : S' \to B$ be the semiample fibration of $K_{S'} + a(S,L)\phi_{*}L$. Pick a sufficiently small $\epsilon > 0$ and perform a relative $(K_{S'} + (a(S,L)-\epsilon)\phi_{*}L)$-MMP over $B$. Then the outcome is a Mori fiber space so one may appeal to the classification of Mori fiber spaces in dimension $2$.
In particular, in this situation $K_{S'} + a(S,L)\phi_{*}L$ vanishes when restricted to a general fiber $F$, yielding the desired description of the $a$-invariant.
\end{proof}

\begin{coro}
Let $S$ be a smooth uniruled projective surface and let $L$ denote a big and nef divisor on $S$.  Then
\begin{equation*}
a(S,L) \in \left\{ \left. \frac{2}{n}  \right| n \in \mathbb{N} \right\} \cup \left\{ \left. \frac{3}{n}  \right| n \in \mathbb{N} \right\}
\end{equation*}
\end{coro}

\begin{proof}
In Case (1) of Proposition \ref{prop:surfaceconstants} we see directly that $a(S,L)$ has the form $2/n$.  In Case (2) of Proposition \ref{prop:surfaceconstants}, $S'$ will admit a curve of anticanonical degree $2$ unless $S' \cong \mathbb{P}^{2}$, in which case $S'$ will admit a curve of anticanonical degree $3$.  For such a curve $C$ we have
\begin{equation*}
a(S,L) = \frac{-K_{S'} \cdot C}{\phi_{*}L \cdot C}
\end{equation*}
and we deduce the desired expression.
\end{proof}

As a consequence we can easily classify the pairs $(S,L)$ with large $a$-invariant.

\begin{theo} \label{theo:largeainvsurfaces}
Let $S$ be a smooth uniruled projective surface and let $L$ denote a big and nef divisor on $S$.  If $a(S,L) > 1$ then
\begin{equation*}
a(S,L) \in \left\{3,2, \frac{3}{2} \right\}
\end{equation*}
Furthermore
\begin{enumerate}
\item If $a(S,L) = 3$ then there is a birational morphism $\phi: S \to \mathbb{P}^{2}$ such that $L = \phi^{*}\mathcal{O}(1)$.
\item If $a(S,L) = 2$ and $(S,L)$ is adjoint rigid then there is a birational morphism $\phi: S \to Q$ such that $Q$ is either a smooth quadric or a quadric cone in $\mathbb{P}^{3}$ and $L = \phi^{*}\mathcal{O}_{\mathbb{P}^{3}}(1)$.
\item If $a(S,L) = 2$ and $(S,L)$ is not adjoint rigid then there is a birational morphism $\phi: S \to S'$ where $S'$ is a ruled surface and $L$ is the pullback of a big and nef divisor with degree $1$ along the fibers of the ruling of $S'$. 
\item If $a(S,L) = 3/2$ then there is a birational morphism $\phi: S \to \mathbb{P}^{2}$ such that $L = \phi^{*}\mathcal{O}(2)$.
\end{enumerate}
\end{theo}

\begin{proof}
Just as in Proposition \ref{prop:surfaceconstants} we run the MMP for $(S,a(S,L)L)$ and repeatedly contract $(-1)$-curves to obtain $\phi: S \to S'$.  Suppose that $E$ is the $(-1)$-curve contracted by the first step of the MMP.  Since $a(S,L) > 1$ and $(K_{S} + a(S,L)L) \cdot E < 0$, we see that $L \cdot E = 0$.  Thus $L$ is pulled back from the target of the first step of the MMP.  Repeating this logic inductively, we see there is some big and nef divisor $L'$ on $S'$ such that $L = \phi^{*}L'$.  Using the classification of weak del Pezzo surfaces we obtain the description of the theorem.
\end{proof}

\subsection{Covers which increase the $a$-invariant} \label{sect:increaseainv}

Suppose that $S$ is a weak del Pezzo surface.  As discussed in the introduction, we expect that the ``pathological'' dominant families of rational curves on $X$ are controlled by generically finite maps $f: Y \to S$ such that $a(Y,-f^{*}K_{S}) > a(S,-K_{S}) = 1$.  Our goal in this section is to classify the situations in which the $a$-invariant of $Y$ is strictly larger than that of $S$.

\begin{theo} \label{theo:dominantaclassificationsurfaces}
Let $S$ be a weak del Pezzo surface and suppose that $f: Y \to S$ is a dominant generically finite morphism such that $a(Y,-f^{*}K_{S}) > a(S,-K_{S})$.  Then we are in one of the following situations:
\begin{enumerate}
\item $(Y,-f^{*}K_{S})$ is not adjoint rigid, %i.e., the Iitaka dimension of $a(Y,-f^{*}K_{S})(-f^*K_S) + K_Y$ is $1$,
$a(Y,-f^{*}K_{S}) = 2$, and the image of a general fiber of the Iitaka fibration for $a(Y,-f^{*}K_{S})(-f^*K_S) + K_Y$ is a curve $C$ on $S$ satisfying $-K_{S} \cdot C = 1$.  In this case $f$ is birationally equivalent to the base-change of a quasi-elliptic fibration by a non-separable map to the base curve.  
\item $\ch(k)=2$, $S$ is a weak del Pezzo surface of degree $2$, and $f$ is birationally equivalent to a purely inseparable morphism of degree $2$ from $\mathbb{P}^{2}$ to the anticanonical model of $S$.
We have $a(Y, -f^*K_S) =  3/2$ in this case.
\item $\ch(k)=2$, $S$ is a weak del Pezzo surface of degree $1$, and $f$ is birationally equivalent to a purely inseparable morphism of degree $2$ from the quadric cone $Q$ to the anticanonical model of $S$.
We have $a(Y, -f^*K_S) = 2$ in this case.
\item $\ch(k)=2$, $S$ is a weak del Pezzo surface of degree $1$, and $f$ is birationally equivalent to a non-separable morphism of degree $4$ from $\mathbb{P}^{2}$ to the anticanonical model of $S$.
We have $a(Y, -f^*K_S) = 3/2$ in this case.
\end{enumerate}
\end{theo}

\begin{proof}
First note that if $f: Y \to S$ is separable then $K_{Y} - f^{*}K_{S}$ is an effective divisor so that $a(Y,-f^{*}K_{S}) \leq a(S,-K_{S})$.  Since we are interested in situations where this inequality fails the map $f$ must be non-separable.

Theorem \ref{theo:largeainvsurfaces} classifies the situations where $a(Y,-f^{*}K_{S}) > 1$.  When $Y$ is not adjoint rigid, the rest of the properties in the first sentence of (1) are immediate from Theorem \ref{theo:largeainvsurfaces}.  Since $S$ carries a $1$-dimensional family of $-K_{S}$-lines, Lemma \ref{lemm:lines} shows that $S$ must have degree $1$ and that the curves $C$ on $S$ are singular members of $|-K_{S}|$.  Resolving this linear series, we see that the fibers of the Iitaka fibration on $Y$ map birationally to the fibers of a quasi-elliptic fibration on the blow-up of $S$.  Furthermore $f$ must be non-separable by Lemma \ref{lemm:ainvseparablecover}.  Altogether this proves the second sentence.  

Next we consider the case when $Y$ is adjoint rigid.  Let $\phi: Y \to Y'$ be the map to a weak del Pezzo surface constructed by Theorem \ref{theo:largeainvsurfaces} by running the $(K_{Y} -a(Y, -f^*K_S) f^{*}K_{S})$-MMP. 
Since each $(-1)$-curve we contract while running the MMP will have vanishing intersection against the pushforward of $-f^{*}K_{S}$, we see that $-f^{*}K_{S} = \phi^{*}L'$ for some divisor $L'$ on $Y'$.  This implies that any $\phi$-exceptional curve on $Y$ is either contracted by $f$ or is mapped to a $(-2)$-curve on $S$.  If we let $\psi: S \to S'$ denote the contraction of all the $(-2)$-curves on $S$, then there is a generically finite morphism $f': Y' \to S'$ forming a commuting diagram
\begin{equation*}
\xymatrix{ Y \ar[d]^{\phi} \ar[r]^{f} &  S \ar[d]^{\psi}  \\
Y' \ar[r]^{f'} & S' }
\end{equation*}
From the equation
\begin{equation*}
K_{Y'} \sim_{\mathbb{Q}} a(Y,-f^{*}K_{S})\phi_{*}f^{*}K_{S} \sim_{\mathbb{Q}} a(Y,-f^{*}K_{S})\phi_{*}f^{*}\psi^{*}K_{S'}
\end{equation*}
we see that $K_{Y'} \sim_{\mathbb{Q}} a(Y,-f^{*}K_{S})f'^{*}K_{S'}$.

Theorem \ref{theo:largeainvsurfaces} shows that there are three types of adjoint rigid surfaces with $a$-invariant larger than $1$.  We argue separately for each case:
\begin{itemize}
\item Case 1: $a(Y,-f^{*}K_{S}) = 3$.  Theorem \ref{theo:largeainvsurfaces} shows that there is a birational morphism $g: Y' \to \mathbb{P}^{2}$ such that $-f'^{*}K_{S'} \sim g^{*}H$ where $H$ is the hyperplane class on $\mathbb{P}^{2}$.  As explained above this divisor is also proportional to $K_{Y'}$.  Thus the only possibility is that $g$ is an isomorphism, $Y' \cong \mathbb{P}^{2}$, and $-f'^{*}K_{S'} \sim H$.  Then $\deg(f') \cdot (-K_{S'})^{2} = (-f'^{*}K_{S'})^{2} = 1$ so that $f'$ is birational, a contradiction.

\item Case 2: $a(Y,-f^{*}K_{S}) = 2$.  Theorem \ref{theo:largeainvsurfaces} shows that there is a birational morphism $g: Y' \to T$ where $T \cong \mathbb{P}^{1} \times \mathbb{P}^{1}$ or the quadric cone $Q$ such that $-f'^{*}K_{S'} \sim g^{*}H$ where $H$ is the restriction of the hyperplane class on $\mathbb{P}^{3}$.  As explained above this divisor is also proportional to $K_{Y'}$.  Thus we must have either $Y' \cong \mathbb{P}^{1} \times \mathbb{P}^{1}$ or $\mathbb{F}_{2}$ and $-2f'^{*}K_{S'} \sim K_{Y'}$. 
When $Y'$ is $\mathbb F_2$, we replace $Y'$ by the quadric cone $Q$. 
We see that $\deg(f') \cdot (-K_{S'})^{2} = (-f'^{*}K_{S'})^{2} = 2$.  The only possibility is that $f'$ has degree $2$ and that $S'$ is a singular degree $1$ weak del Pezzo.  In particular, our ground field must have characteristic $2$.

Suppose for a contradiction that $Y = \mathbb{P}^{1} \times \mathbb{P}^{1}$.  Then the calculation above shows that each family of lines on $Y$ maps to a one-dimensional family of rational curves on $S'$ of anticanonical degree $1$.  Furthermore these two families cannot coincide (since their numerical classes on $S'$ are different).  But by Lemma \ref{lemm:lines} it is only possible for $S$ to carry one such family, showing that such a map cannot exist.

\item Case 3:  $a(Y,-f^{*}K_{S}) = \frac{3}{2}$.  Theorem \ref{theo:largeainvsurfaces} shows that there is a birational morphism $g: Y' \to \mathbb{P}^{2}$ such that $-f'^{*}K_{S'} \sim g^{*}2H$ where $H$ is the hyperplane class on $\mathbb{P}^{2}$.  As explained above this divisor is also proportional to $K_{Y'}$. Thus the only possibility is that $g$ is an isomorphism, $Y' \cong \mathbb{P}^{2}$ and $-f'^{*}K_{S'} \sim 2H$.  Then $\deg(f') \cdot (-K_{S'})^{2} = (-f'^{*}K_{S'})^{2} = 4$.  Thus we see that $S'$ must be a singular weak del Pezzo surface of degree $2$ or $1$ with Picard rank $1$ and that $f'$ must be non-separable and must have degree $2$ or $4$ respectively.  In particular, our ground field must have characteristic $2$.
\end{itemize}
\end{proof}

Using our earlier classification of pathological del Pezzo surfaces, we can give an even more precise description of the possible dominant morphisms which increase the $a$-invariant.

\begin{coro} \label{coro:typesandcases}
Let $S$ be a weak del Pezzo surface and suppose that $f: Y \to S$ is a dominant generically finite morphism such that $a(Y,-f^{*}K_{S}) > a(S,-K_{S})$.
\begin{enumerate}
\item If we are in the setting of Theorem \ref{theo:dominantaclassificationsurfaces}.(1), then $S$ has Type 1.
\item If we are in the setting of Theorem \ref{theo:dominantaclassificationsurfaces}.(2), then $S$ has Type 2.
\item If we are in the setting of Theorem \ref{theo:dominantaclassificationsurfaces}.(3), then $S$ has Type 3.
\item If we are in the setting of Theorem \ref{theo:dominantaclassificationsurfaces}.(4), then $S$ has Type 1 or Type 3.
\end{enumerate}
\end{coro}

\begin{proof}
Theorem \ref{theo:dominantaclassificationsurfaces} gives 4 possible situations.  In Case 1 (respectively Case 2, Case 3) it follows from Claim \ref{clai:type1} (resp.~Claim \ref{clai:type2}, Claim \ref{clai:type3}) that $S$ has Type 1 (resp.~Type 2, Type 3).  It only remains to consider Case 4.

Let $S'$ denote the anticanonical model of $S$ and suppose there is a non-separable degree $4$ morphism $f: \mathbb{P}^{2} \to S'$ such that $f^{*}(-K_{S}) \cong \mathcal{O}(2)$.  Then the image in $S'$ of the lines on $\mathbb{P}^{2}$ yields a $2$-dimensional family of $-K_{S}$-conics.  We then conclude by Corollary \ref{coro:lowdegreeexpectdimanticanonical}.
\end{proof}

To finish off the classification, we make one final remark:

\begin{prop} \label{prop:type3istype1}
The Type 3 surfaces are exactly the same as the Type 1 surfaces in characteristic $2$.
\end{prop}

\begin{proof}
As discussed earlier, a Type 3 surface has Type 1.  Indeed, by definition a Type 3 surface $S$ has an anticanonical model $S'$ which admits a purely inseparable degree $2$ map $f: Q \to S'$ from the quadric cone.  The images of the lines on $Q$ are $-K_{S}$-lines on $S'$, and we conclude that $S$ has Type 1 by Lemma \ref{lemm:lines}.

Conversely, we show that every Type 1 surface in characteristic $2$ has Type 3.  Let $S$ be a Type 1 surface in characteristic $2$, let $\phi: \widetilde{S} \to S$ be the blow-up of the basepoint of $|-K_{S}|$ with exceptional divisor $E$, and let $\pi: \widetilde{S} \to \mathbb{P}^{1}$ be the resolution of the rational map defined by $|-K_{S}|$.  Consider the diagram
\begin{equation*}
\xymatrix{
Y \ar[d]_{p} \ar[r]^{g} &  \widetilde{S} \ar[d]_{\pi} \\
\mathbb{P}^{1} \ar[r]^{F} & \mathbb{P}^{1} }
\end{equation*}
where $F$ denotes the Frobenius map and $Y$ is the normalization of $\widetilde{S} \times_{\mathbb{P}^{1}} \mathbb{P}^{1}$.  Then $g$ is a purely inseparable degree $2$ morphism, $Y$ is smooth, and the general fiber of $p$ is isomorphic to $\mathbb{P}^{1}$.  Set $D = g^{*}E$.  Since $E$ is a section of $\pi$, $D$ is a section of $p$.  Note that $D^{2} = g^{*}E^{2} = -2$ so that $D$ is a $(-2)$-curve on $Y$.

Since the fibers of $p$ have intersection $-1$ against $f^{*}K_{S}$, we see that $a(Y,-f^{*}K_{S}) \geq 2$ where $f: Y \to S$ denote the composition of $g$ and the birational map to $S$.  Theorem \ref{theo:dominantaclassificationsurfaces} then shows that the equality must be attained.  Furthermore, note that $(K_{Y} - 2f^{*}K_{S}) \cdot D = 0$.  We conclude that $K_{Y} - 2f^{*}K_{S}$ is adjoint rigid.  Thus we must be in Case (3) of Theorem \ref{theo:dominantaclassificationsurfaces}.  Corollary \ref{coro:typesandcases} shows that $S$ has Type $3$.
\end{proof}

\subsubsection{Breaking maps and rational curves}

Finally let us remark that the existence of breaking maps implies the existence of families of rational curves with larger than the expected dimension.  In other words, the compatibility we have found between dominant covers with larger $a$-invariant and the presence of families with too large dimension is not just a coincidence.

\begin{prop}
\label{prop:higherthanexpected}
Let $X$ be a smooth weak Fano variety defined over $k$ and let $f : Y \to X$ be a breaking map from a smooth projective variety $Y$.  Suppose that there is a component $M$ of $\overline{M}_{0,0}(Y)$ generically parameterizing a dominant family of rational curves $g : \mathbb{P}^{1} \to Y$ such that $\deg(g^*(K_{Y} - a(Y, -f^*K_X)f^*K_X)) = 0$. Then the family of rational curves on $X$ obtained by applying $f$ to the stable maps in $M$ has higher than expected dimension.  
\end{prop}

\begin{proof}
Since any component of $\overline{M}_{0,0}(Y)$ has at least the expected dimension, we have
\[
\dim \, M \geq \deg(-g^*K_Y) + \dim \, Y -3.
\]
On the other hand since we have $$\deg(-g^{*} K_Y) =  a(Y, -f^*K_X)\deg(-g^{*}f^*K_X)$$ and $a(Y, -f^*K_X) > 1$, we conclude that
\[
\dim \, M  > \deg(-g^*f^*K_X) + \dim \, X -3.
\]
\end{proof}

Since for every breaking map $f: Y \to S$ in Theorem~\ref{theo:dominantaclassificationsurfaces} the surface $Y$ admits infinitely many families of free curves satisfying the assumption of Proposition~\ref{prop:higherthanexpected}, we see that each surface $S$ in the theorem admits infinitely many families of rational curves with higher than the expected dimension.

%\nocite{*}
\bibliographystyle{alpha}
\bibliography{balancedVII}

\end{document}